\documentclass[11pt]{article}

\usepackage{times}
\usepackage{float}

\usepackage{amsfonts,amstext,amsmath,amssymb,amsthm}

\usepackage{chngpage}
\usepackage{rotating}
\usepackage{graphicx}
\usepackage{epstopdf} 

\usepackage{caption}
\usepackage{subcaption}
\usepackage{mathrsfs}

\usepackage{empheq}

\usepackage{paralist}

\long\def\symbolfootnote[#1]#2{\begingroup%
\def\thefootnote{\fnsymbol{footnote}}\footnote[#1]{#2}\endgroup}

\usepackage{hyperxmp}
\usepackage[%
    bookmarks=true,         
    unicode=false,          
    pdftoolbar=true,        
    pdfmenubar=true,        
    hyperfootnotes=false,   
    pdffitwindow=false,     
    pdfstartview={FitH},    
    pdftitle={On the Quasi-Stationary Distribution of the Shiryaev--Roberts Diffusion},    
    pdfauthor={Aleksey S. Polunchenko},     
    pdfsubject={On the Quasi-Stationary Distribution of the Shiryaev--Roberts Diffusion},   
    pdfcreator={Aleksey S. Polunchenko},   
    pdfkeywords={Generalized Shiryaev--Roberts procedure,Quasi-stationary distribution,Quickest change-point detection,Whittaker functions}, 
    pdfnewwindow=true,      
    colorlinks=false,       
    linkcolor=red,          
    citecolor=green,        
    filecolor=magenta,      
    urlcolor=cyan           
]{hyperref}

\usepackage{suffix}
\usepackage{mathtools}

\usepackage[T1]{fontenc}

\usepackage{titlesec} 

\titleformat{\section}{\large\bfseries\uppercase}{\thesection.}{.5em}{}
\titlespacing*{\section}{0pt}{*3}{*2}
\titleformat{\subsection}{\normalfont\bfseries}{\thesubsection.}{.5em}{}
\titlespacing*{\subsection}{0pt}{*3}{*2}
\titleformat{\subsubsection}{\normalfont\bfseries}{\thesubsubsection.}{.5em}{}
\titlespacing*{\subsubsection}{0pt}{*3}{*2}

\numberwithin{equation}{section}

\usepackage[sort&compress]{natbib}

\DeclarePairedDelimiterX\MeijerM[3]{\lparen}{\rparen}%
{\begin{smallmatrix}#1 \\ #2\end{smallmatrix}\delimsize\vert\,#3}

\newcommand\MeijerG[8][]{%
  G^{\,#2,#3}_{#4,#5}\MeijerM[#1]{#6}{#7}{#8}}

\WithSuffix\newcommand\MeijerG*[7]{%
  G^{\,#1,#2}_{#3,#4}\MeijerM*{#5}{#6}{#7}}

\long\def\symbolfootnote[#1]#2{\begingroup%
\def\thefootnote{\fnsymbol{footnote}}\footnote[#1]{#2}\endgroup}

\renewcommand{\Pr}{\mathbb{P}} 
\DeclareMathOperator{\EV}{\mathbb{E}} 
\DeclareMathOperator{\Var}{\mathrm{Var}}

\DeclareMathOperator{\Ei}{Ei}
\DeclareMathOperator{\E1}{E_1}

\newcommand{\T}{T}

\renewcommand{\le}{\leqslant} 
\renewcommand{\ge}{\geqslant}

\newcommand{\abs}[1]{\left\vert#1\right\vert}

\DeclareMathOperator{\One}{\mathchoice{\rm 1\mskip-4.2mu l}{\rm 1\mskip-4.2mu l}{\rm 1\mskip-4.6mu l}{\rm 1\mskip-5.2mu l}}
\newcommand{\indicator}[1]{{\One_{\left\{#1\right\}}}}

\theoremstyle{plain} 
\newtheorem{theorem}{Theorem}[section]
\newtheorem{lemma}{Lemma}[section]

\graphicspath{{./gfxPDF/}}

\begin{document}

\title{\textbf{\Large On the Quasi-Stationary Distribution of the Shiryaev--Roberts Diffusion}}

\date{}
\author{}
\maketitle

\begin{center}
\null\vskip-2cm\author{
\textbf{\large Aleksey\ S.\ Polunchenko}\\
Department of Mathematical Sciences, State University of New York at Binghamton,\\Binghamton, New York, USA
}
\end{center}
%
%
%
%
\symbolfootnote[0]{\normalsize\hspace{-0.6cm}Address correspondence to A.\ S.\ Polunchenko, Department of Mathematical Sciences, State University of New York (SUNY) at Binghamton, 4400 Vestal Parkway East, Binghamton, NY 13902--6000, USA; Tel: +1 (607) 777--6906; Fax: +1 (607) 777--2450; E-mail:~\href{mailto:aleksey@binghamton.edu}{aleksey@binghamton.edu}.}\\
%
%
{\small\noindent\textbf{Abstract:} We consider the diffusion $(R_t^r)_{t\ge0}$ generated by the equation $dR_t^r=dt+\mu R_t^r dB_t$ with $R_0^r\triangleq r\ge0$ fixed, and where $\mu\neq0$ is given, and $(B_t)_{t\ge0}$ is standard Brownian motion. We assume that $(R_t^r)_{t\ge0}$ is stopped at $\mathcal{S}_A^r\triangleq\inf\{t\ge0\colon R_t^r=A\}$ with $A>0$ preset, and obtain a closed-from formula for the quasi-stationary distribution of $(R_t^r)_{t\ge0}$, i.e., the limit $Q_A(x)\triangleq\lim_{t\to+\infty}\Pr(R_t^r\le x|\mathcal{S}_A^r>t)$, $x\in[0,A]$. Further, we also prove $Q_A(x)$ to be unimodal for any $A>0$, and obtain its entire moment series. More importantly, the pair $(\mathcal{S}_A^r,R_t^r)$ with $r\ge0$ and $A>0$ is the well-known Generalized Shiryaev--Roberts change-point detection procedure, and its characteristics for $r\sim Q_A(x)$ are of particular interest, especially when $A>0$ is large. In view of this circumstance we offer an order-three large-$A$ asymptotic approximation of $Q_A(x)$ valid for all $x\in[0,A]$. The approximation is rather accurate even if $A$ is lower than what would be considered ``large'' in practice.
}
\\ \\
%
%
{\small\noindent\textbf{Keywords:} {Shiryaev--Roberts procedure; Quasi-stationary distribution; Quickest change-point detection; Whittaker functions.}
\\ \\
%
%
%
%
%
%
%
%
%

\section{Introduction} 
\label{sec:intro}

The general theme of this work is quickest change-point detection. The subject is concerned with the design and analysis of dependable ``watchdog''-type statistical procedures for early detection of unanticipated changes that may (or may not) occur online in the characteristics of a ``live''-monitored process. See, e.g.,~\cite{Shiryaev:Book78}, \cite{Basseville+Nikiforov:Book93}, \cite{Poor+Hadjiliadis:Book09}, \cite{Veeravalli+Banerjee:AP2013}, and~\cite[Part~II]{Tartakovsky+etal:Book2014}. A change-point detection procedure is a stopping time, $\T$, that is adapted to the filtration, $(\mathcal{F}_t)_{t\ge0}$, generated by the observed process, $(X_t)_{t\ge0}$; the interpretation of $\T$ is that it is a rule to stop and declare that the statistical profile of the observed process may have (been) changed. A ``good'' (i.e., optimal or nearly optimal) detection procedure is one that minimizes (or nearly minimizes) the desired detection delay penalty-function, subject to a constraint on the ``false alarm'' risk. For an overview of the major optimality criteria see, e.g.,~\cite{Tartakovsky+Moustakides:SA10}, \cite{Polunchenko+Tartakovsky:MCAP2012},~\cite{Veeravalli+Banerjee:AP2013}, and~\cite[Part~II]{Tartakovsky+etal:Book2014}.

This work is motivated by the classical minimax change-point detection problem where the observed process, $(X_t)_{t\ge0}$, is standard Brownian motion that at an unknown (nonrandom) time moment $\nu\ge0$---referred to as the change-point---may (or may not) experience an abrupt and permanent change in the drift, from a value of zero initially, i.e., $\EV[X_t]=0$ for $t\in[0,\nu]$, to a known value $\mu\neq0$ following the change-point, i.e., $\EV[X_t]=\mu t$ for $t>\nu$. The goal is to find out---as quickly as is possible within an {\it a~priori} set level of the ``false positive'' risk---whether the drift of the process is no longer zero. See, e.g.,~\cite{Pollak+Siegmund:B85},~\cite{Shiryaev:RMS1996,Shiryaev:Bachelier2002},~\cite{Moustakides:AS2004},~\cite{Shiryaev:MathEvents2006},~\cite{Feinberg+Shiryaev:SD2006},~\cite{Burnaev+etal:TPA2009},~\cite[Chapter~5]{Shiryaev:Book2011},~\cite{Polunchenko+Sokolov:MCAP2016} and~\cite{Polunchenko:SA2016}. More formally, under the Brownian motion change-point scenario, the observed process, $(X_t)_{t\ge0}$, is governed by the stochastic differential equation (SDE):
\begin{align}\label{eq:BM-change-point-model}
dX_t
&=
\mu\indicator{t>\nu}dt+dB_t,\;t\ge0,\;\text{with}\;X_0=0,
\end{align}
where $(B_t)_{t\ge0}$ is standard Brownian motion (i.e., $\EV[dB_t]=0$, $\EV[(dB_t)^2]=dt$, and $B_0=0$), $\mu\neq0$ is the {\em known} post-change drift value, and $\nu\in[0,\infty]$ is the unknown (nonrandom) change-point; here and onward, the notation $\nu=0$ ($\nu=\infty$) is to be understood as the case when the drift is in effect {\it ab initio} (or never, respectively). Let $\Pr_{\infty}$ ($\Pr_{0}$) denote the probability measure (distribution law) generated by the observed process, $(X_t)_{t\ge0}$, under the assumption that $\nu=\infty$ ($\nu=0$); note that $\Pr_{\infty}$ is the Wiener measure.

A sensible way to perform change-point detection under model~\eqref{eq:BM-change-point-model} is to use the Generalized Shiryaev--Roberts (GSR) procedure proposed by~\cite{Moustakides+etal:SS11} as a headstarted (hence, more general) version of the classical quasi-Bayesian Shiryaev--Roberts (SR) procedure that emerged from the independent work of Shiryaev~\citeyearpar{Shiryaev:SMD61,Shiryaev:TPA63} and that of Roberts~\citeyearpar{Roberts:T66}. Specifically, tailored to the Brownian motion scenario~\eqref{eq:BM-change-point-model}, the GSR procedure is given by the stopping time:
\begin{align}\label{eq:T-GSR-def}
\mathcal{S}_A^r
&\triangleq
\inf\{t\ge0\colon R_{t}^r=A\}\;\text{such that}\;\inf\{\varnothing\}=\infty,
\end{align}
where $A>0$ is the detection threshold (set in advance in such a way that the ``false positive'' risk is within the desired margin of tolerance), and the GSR statistic $(R_{t}^r)_{t\ge0}$ is the diffusion process that satisfies the stochastic differential equation
\begin{align}\label{eq:Rt_r-def}
dR_t^r
&=
dt+\mu R_t^r dX_t
\;\text{with}\;
R_0^r\triangleq r\ge0,
\end{align}
with $(X_t)_{t\ge0}$ as in~\eqref{eq:BM-change-point-model}; the initial value $r$ is known as the headstart.

The choice to go with the GSR procedure may be justified by the result previously obtained (for the discrete-time analogue of the problem) by~\cite{Tartakovsky+etal:TPA2012} where the GSR procedure with a carefully designed headstart was shown to be nearly minimax in the sense of Pollak~\citeyearpar{Pollak:AS85}; an attempt to generalize this result to the Brownian motion scenario~\eqref{eq:BM-change-point-model} was made, e.g., by~\cite{Burnaev:ARSAIM2009}. The GSR procedure's near Pollak-minimaxity is a strong optimality property known in the literature as {\em order-three} minimaxity because the respective delay cost is minimized up to an additive term that goes to zero together with the ``false alarm'' risk level. Also, it was demonstrated explicitly by~\cite{Tartakovsky+Polunchenko:IWAP10} and by~\cite{Polunchenko+Tartakovsky:AS10} that in two specific (discrete-time) scenarios the GSR procedure (again with a ``finetuned'' headstart) is actually {\em exactly} minimax in Pollak's~\citeyearpar{Pollak:AS85} sense. More importantly, while a {\em general} solution to Pollak's~\citeyearpar{Pollak:AS85} minimax change-point detection problem is still unknown, there is a universal ``recipe'' (also proposed by Pollak~\citeyear{Pollak:AS85}) to achieve strong order-three Pollak-minimaxity, and the main ingredient of the ``recipe'' is the GSR procedure. Specifically, the idea is to start the GSR statistic off a random number sampled from the statistic's so-called quasi-stationary distribution (formally defined below). Pollak~\citeyearpar{Pollak:AS85} proved that, in the discrete-time setup, such a randomized ``tweak'' of the GSR procedure is order-three Pollak-minimax; see also~\cite{Tartakovsky+etal:TPA2012}. The same may well hold true for the Brownian motion scenario~\eqref{eq:BM-change-point-model} too, although, to the best of our knowledge, this question has not yet been investigated in the literature, except in the work of~\cite{Burnaev+etal:TPA2009} where the GSR procedure was shown to be almost Pollak-minimax, but only up to the second order (the delay cost is minimized up to an additive term that goes to a constant as the ``false alarm'' risk level goes to zero). Even though Pollak's~\citeyearpar{Pollak:AS85} idea to randomize the GSR procedure as described above {\em does not} necessarily lead to {\em strict} Pollak-minimaxity (see~\citealt{Tartakovsky+Polunchenko:IWAP10,Polunchenko+Tartakovsky:AS10} for counterexamples), it does allow to achieve order-three minimaxity in the discrete-time setup, and whether or not this is also the case for the continuous-time scenario~\eqref{eq:BM-change-point-model} is a problem that is certainly worthy of consideration. This work is an attempt to make the first step in this direction. Incidentally, the quasi-stationary distribution is also essential for the evaluation of the GSR procedure's minimax performance in the stationary regime, i.e., when $\nu\to+\infty$; see~\cite{Pollak+Siegmund:B85}.

More concretely, the overall aim of this paper is to obtain an exact closed-form formula for the GSR statistic's~\eqref{eq:Rt_r-def} quasi-stationary distribution and an accurate asymptotic approximation thereof for when the GSR procedure's detection threshold $A>0$ is large. Formally, the sought quasi-stationary distribution is defined as
\begin{align}\label{eq:QSD-def}
Q_A(x)
&\triangleq
\lim_{t\to+\infty}\Pr_{\infty}(R_t^r\le x|\mathcal{S}_A^r>t)
\;\text{with}\;
q_A(x)
\triangleq
\frac{d}{dx}Q_A(x),
\end{align}
which is independent of the headstart $R_0^r\triangleq r\ge0$, provided $r\in[0,A]$, i.e., within the support of the density $q_A(x)$. The existence of this distribution for the Brownian motion scenario~\eqref{eq:BM-change-point-model} follows from the results previously obtained in the fundamental work of Mandl~\citeyearpar{Mandl:CMJ1961}; see also, e.g.,~\cite{Cattiaux+etal:AP2009}. While we do get $Q_A(x)$ and $q_A(x)$ expressed analytically and explicitly (see Section~\ref{sec:main}), the obtained formulae involve special functions, and, as a result, a {\em precise} performance analysis of Pollak's~\citeyearpar{Pollak:AS85} randomized GSR procedure is problematic: the calculus involved is prohibitively difficult. The usual way around this is to assume the GSR procedure's detection threshold $A>0$ is large and study the randomized GSR procedure asymptotically as $A\to+\infty$. To that end, it is known (see~\citealt{Pollak+Siegmund:B85,Pollak+Siegmund:JAP1996}) that $Q_A(x)$ does have a limit as $A\to+\infty$, and the limit is the GSR statistic's so-called stationary distribution defined as
\begin{align}\label{eq:ST-cdf-def}
H(x)
&\triangleq
\lim_{t\to+\infty}\Pr_{\infty}(R_t^r\le x)
\;\text{with}\;
h(x)
\triangleq
\frac{d}{dx}H(x),
\;\text{where}\;
x\in[0,+\infty),
\end{align}
which is also independent of the headstart $R_0^r\triangleq r\ge0$; the convergence of $Q_A(x)$ to $H(x)$ as $A\to+\infty$ is pointwise, at all continuity points of $H(x)$. Although the stationary distribution $H(x)$ has already been found and studied in the literature (see, e.g.,~\citealt{Shiryaev:SMD61,Shiryaev:TPA63},~\citealt{Pollak+Siegmund:B85},~\citealt{Feinberg+Shiryaev:SD2006}, or~\citealt{Burnaev+etal:TPA2009}), a ``little oh''-level investigation of the randomized GSR procedure's characteristics requires a more ``fine'' large-$A$ asymptotic approximation of the quasi-stationary distribution itself. To that end, we offer a large-$A$ {\em order-three} expansion of the density $q_A(x)$ valid for all $x\in[0,A]$. The expansion is derived directly from the exact formula for $q_A(x)$, with the aid of the Mellin integral transform, which, if need be, can also be used to expand $q_A(x)$ even further, beyond the third-order term. As an auxiliary result, we prove that the quasi-stationary distribution is unimodal for any $A>0$; it is of note that its limit as $A\to+\infty$, i.e., the stationary distribution~\eqref{eq:ST-cdf-def}, is known to be unimodal as well. We also obtain the quasi-stationary distribution's entire moment series.

The remainder of the paper is organized as follows. Section~\ref{sec:preliminaries} fixes nomenclature and notation, and provides the necessary preliminary background. The main contribution---i.e., the quasi-stationary distribution and its properties---is the subject of Section~\ref{sec:main}. Conclusions and outlook follow in Section~\ref{sec:conclusion}. Appendix offers proofs of certain lemmas.

\section{Preliminaries}
\label{sec:preliminaries}

For notational brevity, we shall henceforth omit the subscript ``$A$'' in ``$Q_A(x)$'' as well as in ``$q_A(x)$'', unless the dependence on $A$ is noteworthy. Also, for technical convenience and without loss of generality, we shall primarily deal with $q(x)$ rather than with $Q(x)$.

It is has already been established in the literature (see, e.g.,~\citealt{Mandl:CMJ1961},~\citealt{Cattiaux+etal:AP2009}, or~\citealt[Equation~(35),~p.~528]{Burnaev+etal:TPA2009}) that $q(x)$, formally defined in~\eqref{eq:QSD-def}, is the solution of a certain boundary-value problem composed of a second-order ordinary differential equation (ODE) considered on the interval $[0,A]$, a pair of boundary conditions, and a normalization constraint. Specifically, the ODE---which we shall henceforth refer to as the {\em master equation}---is of the form
\begin{align}\label{eq:master-eqn}
\dfrac{\mu^2}{2}\dfrac{d^2}{dx^2}\big[x^2\,q(x)\big]
-
\dfrac{d}{dx}\big[q(x)\big]
&=
\lambda\,q(x),
\;\;x\in[0,A],
\end{align}
where $\lambda^{(0)}$ is the {\em dominant} eigenvalue of the differential operator
\begin{align}\label{eq:Doperator-def}
\mathscr{D}
&\triangleq
\dfrac{\mu^2}{2}\dfrac{\partial^2}{\partial x^2} x^2
-
\dfrac{\partial}{\partial x},
\end{align}
i.e., the infinitesimal generator of the GSR diffusion $(R_t^r)_{t\ge0}$ under the $\Pr_{\infty}$ probability measure; observe from~\eqref{eq:BM-change-point-model} and~\eqref{eq:Rt_r-def} that the $\Pr_{\infty}$-differential of the GSR diffusion $(R_t^r)_{t\ge0}$ is $dR_t^r=dt+\mu R_t^r dB_t$ where $R_0^r\triangleq r\in[0,A]$ and $A>0$. It goes without saying that $\lambda$ (or any other eigenvalue of the operator $\mathscr{D}$ for that matter) is dependent on $A$, and, wherever necessary, we shall emphasize this dependence via the notation $\lambda_{A}$.

Next, the behavior of $q(x)$ near the left end-point of the interval $[0,A]$---which is the GSR diffusion's domain---must be such that
\begin{align}\label{eq:bnd-cond-0}
\lim_{x\to0+}\left\{\dfrac{\mu^2}{2}\dfrac{\partial}{\partial x}\big[x^2\,q(x)\big]-q(x)\right\}
&=
0,
\end{align}
which is to say that $x=0$ is, in Feller's~\citeyearpar{Feller:AM1952} classification, an {\em entrance} boundary for the process $(R_t^r)_{t\ge0}$; in ``differential equations speak'', this is a Neumann-type boundary condition. The boundary condition at the other end of the interval $[0,A]$ is of the form
\begin{align}\label{eq:bnd-cond-A}
q(A)
&=
0,
\end{align}
i.e., the density $q(x)$ must vanish at $x=A$, for, by definition~\eqref{eq:T-GSR-def} of the GSR stopping time, the GSR diffusion is ``killed'' or ``absorbed'' once it hits the detection threshold $A>0$; in ``differential equations speak'', this is a Dirichlet-type boundary condition.

Lastly, the density $q(x)$ must also satisfy the obvious normalization constraint
\begin{align}\label{eq:norm-constraint}
\int_{0}^{A}q(x)\,dx
&=
1.
\end{align}

Subject to boundary conditions~\eqref{eq:bnd-cond-0} and~\eqref{eq:bnd-cond-A}, equation~\eqref{eq:master-eqn} is a Sturm--Liouville problem. It is straightforward to see that by virtue of the multiplying factor
\begin{align}\label{eq:speed-measure}
\mathfrak{m}(x)
&\triangleq
\dfrac{2}{\mu^2 x^2}\,e^{-\tfrac{2}{\mu^2 x}}
\end{align}
the equation can be brought to the canonical Sturm--Liouville form
\begin{align}\label{eq:eigen-eqn-two}
\dfrac{\mu^2}{2}\dfrac{d}{dx}\left[x^2\,\mathfrak{m}(x)\,\dfrac{d}{dx}\varphi(x)\right]
&=
\lambda\,\mathfrak{m}(x)\,\varphi(x),
\end{align}
where the unknown function $\varphi(x)$ is such that $q(x)\propto\mathfrak{m}(x)\,\varphi(x)$, i.e., $q(x)$ is a multiple of $\mathfrak{m}(x)\,\varphi(x)$. Hence, our problem effectively is to consider the Sturm--Liouville operator
\begin{align}\label{eq:Goperator-def}
\mathscr{G}
&\triangleq
\dfrac{\mu^2}{2\,\mathfrak{m}(x)}\dfrac{d}{dx}x^2\,\mathfrak{m}(x)\,\dfrac{d}{dx}
\end{align}
with $\mathfrak{m}(x)$ given by~\eqref{eq:speed-measure}, restrict it to the interval $x\in[0,A]$, and recover its dominant eigenvalue $\lambda\equiv\lambda_A$ and the respective eigenfunction $\varphi(x)\equiv\varphi(x,\lambda)$ for which
\begin{align}\label{eq:psi-bdd-conds}
\lim_{x\to0+}\left[x^2\,\mathfrak{m}(x)\,\dfrac{d}{dx}\varphi(x,\lambda)\right]
&=
0
\;\;\text{and}\;\;
\varphi(A,\lambda)
=
0.
\end{align}

The general theory of second-order differential operators or Sturm--Liouville operators (such as our operators $\mathscr{D}$ and $\mathscr{G}$ defined above) is well-developed, and, in particular, the spectral properties of such operators are well-understood. The classical fundamental references on the subject are~\cite{Titchmarsh:Book1962},~\cite{Levitan:Book1950},~\cite{Coddington+Levinson:Book1955},~\cite{Dunford+Schwartz:Book1963}, and~\cite{Levitan+Sargsjan:Book1975}; for applications of the theory to diffusion processes, see, e.g.,~\cite[Section~4.11]{Ito+McKean:Book1974}, and~\cite{Linetsky:HandbookChapter2007} who provides a great overview of the state-of-the-art in the field. We now recall a few results from the general Sturm--Liouville theory that directly apply to our specific Sturm--Liouville problem. These results will be utilized in the sequel, and, conveniently enough, all of them (and much more) can be found in the work of~\cite{Fulton+etal:FIT-TR1999}; see also, e.g.,~\cite{Kent:ZWVG1980}.

We start with~\cite[Theorem~18,~p.~22]{Fulton+etal:FIT-TR1999} which asserts that our Sturm--Liouville problem given by equation~\eqref{eq:eigen-eqn-two} and two boundary conditions~\eqref{eq:psi-bdd-conds} does have a solution $\varphi(x,\lambda)$ defined up to an arbitrary multiplicative factor independent of $x$. Consequently, by invoking the normalization constraint~\eqref{eq:norm-constraint}, one can conclude that the quasi-stationary pdf $q(x)$ does exist and is unique.

Next, we turn to~\cite[Section~7]{Fulton+etal:FIT-TR1999} which introduces ten mutually exclusive categories to classify Sturm--Liouville operators depending on their spectral properties. Our Sturm--Liouville problem belongs to Spectral Category 1 (see~\citealt[p.~22]{Fulton+etal:FIT-TR1999}). This means, among many things, that the spectrum of the operator $\mathscr{D}$ given by~\eqref{eq:Doperator-def} is purely discrete, and is determined entirely by the Dirichlet boundary condition~\eqref{eq:bnd-cond-A}. More concretely, the spectrum is the set of solutions of the equation $\varphi(A,\lambda)=0$ where $\lambda$ is the unknown and $A>0$ is fixed. This equation has countably many distinct zeros, say $\{\lambda^{(i)}\}_{i\ge0}$, and each one of them is simple, and their series for any fixed $0<A<+\infty$ is such that $0<-\lambda^{(0)}<-\lambda^{(1)}<-\lambda^{(2)}<\cdots$ with $-\lim_{i\to+\infty}\lambda^{(i)}=+\infty$. Furthermore, for the dominant eigenvalue $\lambda_A\;(\equiv\lambda^{(0)})$ the equation $\varphi(x,\lambda_A)=0$ with $x$ being the unknown has no zeros inside the interval $x\in(0,A)$. All this translates to the following two results.
\begin{lemma}\label{lem:qA-nozeros}
For any fixed $A>0$ the quasi-stationary pdf $q_A(x)$ has no zeros inside the interval $(0,A)$, i.e., $q_A(x)>0$ for all $x\in(0,A)$.
\end{lemma}
\begin{lemma}\label{lem:Dop-eigval-neg}
For any fixed $A>0$ the dominant eigenvalue $\lambda\equiv\lambda_A$ of the operator $\mathscr{D}$ given by~\eqref{eq:Doperator-def} is nonpositive, i.e., $\lambda_A\le0$.
\end{lemma}
\begin{proof}
The claim is actually true for {\em any} eigenvalue of the operator $\mathscr{D}$. For an explicit proof adapted to the context of diffusion processes, see, e.g.,~\cite[Section~1,~pp.~265--266]{Wong:SPMPE1964} or~\cite[Section~3,~pp.~125--126]{Polunchenko:SA2016}.
\end{proof}
It is of note that $\lambda=0$ and $\varphi(x,0)\equiv1$ is always a (nontrivial) solution of~\eqref{eq:eigen-eqn-two}, and $q(x)$ in this case is a multiple of $\mathfrak{m}(x)$ given by~\eqref{eq:speed-measure}. However, while $q(x)\propto\mathfrak{m}(x)$ does (trivially) satisfy the entrance boundary condition~\eqref{eq:bnd-cond-0} for any $A$, the absorbing boundary condition~\eqref{eq:bnd-cond-A} is satisfied only in the limit as $A\to+\infty$. Moreover, since $\mathfrak{m}(x)\ge0$ for all $x\in\mathbb{R}$ and because it integrates to unity over $x\in[0,+\infty)$, the normalization constraint~\eqref{eq:norm-constraint} is automatically fulfilled for $q(x)=\mathfrak{m}(x)$ in the limit as $A\to+\infty$. We therefore arrive at the result first obtained by Shiryaev~\citeyearpar{Shiryaev:SMD61,Shiryaev:TPA63} that $q(x)=\mathfrak{m}(x)$ with $x\in[0,+\infty)$ is the stationary distribution~\eqref{eq:ST-cdf-def} of the GSR statistic, i.e., $h(x)\triangleq\lim_{A\to+\infty}q_A(x)=\mathfrak{m}(x)$, $x\in[0,+\infty)$; see also, e.g.,~\cite{Pollak+Siegmund:B85},~\cite{Feinberg+Shiryaev:SD2006},~\cite[Remark~2,~p.~529]{Burnaev+etal:TPA2009}, or~\cite{Polunchenko+Sokolov:MCAP2016}. This momentless distribution is a special case of the (extreme-value) Fr\'{e}chet distribution.

We conclude this section by recalling that, according to the well-known basic result (see, e.g.,~\citealt[Lemma~1.1.1]{Levitan:Book1950} or~\citealt[Lemma~1.1,~pp.~2--3]{Levitan+Sargsjan:Book1975}) from the general Sturm--Liouville theory, the eigenfunctions $\varphi(x,\lambda)$ indexed by $\lambda$ of the operator $\mathscr{G}$ form an orthonormal basis in the Hilbert space $L^2([0,A],\mathfrak{m})$ of real-valued ``$\mathfrak{m}(x)$''-measurable, square-integrable (with respect to the ``$\mathfrak{m}(x)$''-measure) functions defined on the interval $[0,A]$ equipped with the ``$\mathfrak{m}(x)$''-weighted inner product:
\begin{align*}
\langle f,g\rangle_{\mathfrak{m}}
&\triangleq
\int_{0}^{A}\mathfrak{m}(x)\,f(x)\,g(x)\,dx.
\end{align*}

More specifically, the foregoing means that if $\varphi(x,\lambda^{(i)})$ and $\varphi(x,\lambda^{(j)})$ are any two eigenfunctions of the operator $\mathscr{G}$ given by~\eqref{eq:Goperator-def}, then
\begin{align*}
\int_{0}^{A}\mathfrak{m}(x)\,\varphi(x,\lambda^{(i)})\,\varphi(x,\lambda^{(j)})\,dx
&=
\indicator{i=j},
\end{align*}
where $\varphi(x,\lambda^{(i)})$ and $\varphi(x,\lambda^{(j)})$ are each assumed to be of unit ``length'', i.e., $\|\varphi(\cdot,\lambda^{(i)})\|=1=\|\varphi(\cdot,\lambda^{(j)})\|$, with the ``length'' defined as
\begin{align}\label{eq:eigfun-norm-def}
\|\varphi(\cdot,\lambda)\|^2
&\triangleq
\int_{0}^{A}\mathfrak{m}(x)\,\varphi^2(x,\lambda)\,dx.
\end{align}

We are now in a position to attack the master equation~\eqref{eq:master-eqn} directly. This is the object of the next section, which is the main section of this paper.

\section{The Quasi-stationary Distribution and Its Characteristics}
\label{sec:main}

The plan now is to first fix the detection threshold $A>0$ and solve the master equation~\eqref{eq:master-eqn} analytically to recover $q_A(x)$ and $Q_A(x)$, both in a closed form and foro all $x\in[0,A]$, and then assume that $A$ is large and consider the asymptotic (as $A\to+\infty$) behavior of $q_A(x)$ on its support $x\in[0,A]$. Much of the exact analytical solution part is a build-up to the closely related earlier work of the author~\citeyearpar{Polunchenko:SA2016}; see also~\cite{Linetsky:OR2004} and~\cite{Polunchenko+Sokolov:MCAP2016}.

\subsection{The Exact Formulae and Analysis}

The idea is to apply the change of variables
\begin{align}\label{eq:ux-var-def}
x&\mapsto u=u(x)=\dfrac{2}{\mu^2 x},\;\text{so that}\;u\mapsto x=x(u)=\dfrac{2}{\mu^2 u}\;\text{and}\;\dfrac{dx}{x}=-\dfrac{du}{u},
\end{align}
along with the substitution
\begin{align}\label{eq:eigfun-substitution}
\varphi(x)&\mapsto\varphi(u)\triangleq\dfrac{v(u)}{\sqrt{\mathfrak{m}(u)}}=\sqrt{\dfrac{\mu^2}{2}}\dfrac{1}{u}\,e^{\tfrac{u}{2}}\,v(u)\propto \dfrac{1}{u}\,e^{\tfrac{u}{2}}\,v(u),
\end{align}
to bring our equation to the form
\begin{align}\label{eq:eigenfcn-eqn-whit-form}
v_{uu}(u)+\left\{-\frac{1}{4}+\frac{1}{u}+\frac{1/4-\xi^2/4}{u^2}\right\}v(u)
&=
0,
\end{align}
where
\begin{align}\label{eq:xi-def}
\xi
&\equiv
\xi(\lambda)
\triangleq
\sqrt{1+\dfrac{8}{\mu^2}\lambda}
\;\;\text{so that}\;\;
\lambda
\equiv
\lambda(\xi)
=
\dfrac{\mu^2}{8}(\xi^2-1),
\end{align}
and one may see that $\xi$ is either purely real (ranging between 0 and 1) or purely complex, because $\lambda\le0$, as asserted by Lemma~\ref{lem:Dop-eigval-neg} above. We also remark that equation~\eqref{eq:eigenfcn-eqn-whit-form} is symmetric with respect to the sign of $\xi$.

The obtained equation~\eqref{eq:eigenfcn-eqn-whit-form} is a particular version of the classical Whittaker~\cite{Whittaker:BAMS1904} equation
\begin{align}\label{eq:Whittaker-eqn}
w_{zz}(z)+\left\{-\dfrac{1}{4}+\dfrac{a}{z}+\dfrac{1/4-b^2}{z^2}\right\}w(z)
&=
0,
\end{align}
where $w(z)$ is the unknown function of $z\in\mathbb{C}$, and $a,b\in\mathbb{C}$ are two given parameters; see, e.g.,~\cite[Chapter~I]{Buchholz:Book1969}. A self-adjoint homogeneous second-order ODE, Whittaker's equation~\eqref{eq:Whittaker-eqn} provides a definition for the well-known two Whittaker functions: the latter may be defined as the equation's two fundamental solutions. More concretely, the two Whittaker functions are special functions that are conventionally denoted $W_{a,b}(z)$ and $M_{a,b}(z)$, where the indices $a$ and $b$ are the Whittaker's equation~\eqref{eq:Whittaker-eqn} parameters. The functions are typically considered in the cut plane $|\arg(z)\,|<\pi$ to ensure they are not multi-valued. For an extensive study of these functions and various properties thereof, see, e.g.,~\cite{Buchholz:Book1969} or~\cite{Slater:Book1960}.

By combining~\eqref{eq:Whittaker-eqn},~\eqref{eq:eigenfcn-eqn-whit-form},~\eqref{eq:eigfun-substitution} and~\eqref{eq:ux-var-def} one can see that the {\em general} form of $\varphi(x,\lambda)$ is as follows:
\begin{align}\label{eq:eigfun-gen-form}
\varphi(x,\lambda)
&=
\dfrac{\mu^2 x}{2}\,e^{\tfrac{1}{\mu^2 x}}\left\{C_1M_{1,\tfrac{\xi(\lambda)}{2}}\left(\dfrac{2}{\mu^2 x}\right)+C_2W_{1,\tfrac{\xi(\lambda)}{2}}\left(\dfrac{2}{\mu^2 x}\right)\right\},
\end{align}
where $C_1$ and $C_2$ are arbitrary constants such that $C_1^2+C_2^2\neq0$. The obvious next step is to ``pin down'' the two constants so as to make the obtained general $\varphi(x,\lambda)$ satisfy the boundary conditions~\eqref{eq:psi-bdd-conds} as well as make the pdf $q(x)\propto\mathfrak{m}(x)\,\varphi(x,\lambda)$ where $\mathfrak{m}(x)$ is given by~\eqref{eq:speed-measure} satisfy the normalization constraint~\eqref{eq:norm-constraint}. To that end, it follows from~\cite{Polunchenko:SA2016} that $C_1$ must be set to zero to ensure the pdf $q(x)$ fulfills the entrance boundary condition~\eqref{eq:bnd-cond-0}. Hence, the general form~\eqref{eq:eigfun-gen-form} of $\varphi(x,\lambda)$ simplifies down to
\begin{align}\label{eq:eigfun-gen-form2}
\varphi(x,\lambda)
&=
C\,\dfrac{\mu^2 x}{2}\,e^{\tfrac{1}{\mu^2 x}} W_{1,\tfrac{\xi(\lambda)}{2}}\left(\dfrac{2}{\mu^2 x}\right),
\end{align}
where $C\neq0$ is an arbitrary constant. Assuming for the moment that $\lambda$ is known, the foregoing is sufficient to express $q(x)$ explicitly. Specifically, with the aid of~\cite[Integral~7.623.7,~p.~824]{Gradshteyn+Ryzhik:Book2007} which states that
\begin{align}\label{eq:Whit-int-formula}
\begin{split}
\int_{1}^{+\infty}(x-1)^{c-1}&x^{a-c-1}\, e^{-\tfrac{zx}{2}}\,W_{a,b}(zx)\,dx
=\\
&\qquad=
\Gamma(c)\,e^{-\tfrac{z}{2}}\,W_{a-c,b}(z),
\;\text{provided $\Re(c)>0$ and $\Re(z)>0$},
\end{split}
\end{align}
where here and onward $\Gamma(z)$ denotes the Gamma function (see, e.g.,~\citealt[Chapter~6]{Abramowitz+Stegun:Handbook1964}) one gets
\begin{align*}
\int_{0}^{A}\mathfrak{m}(x)\,\varphi(x,\lambda)\,dx
&=
C\,e^{-\tfrac{1}{\mu^2 A}}\,W_{0,\tfrac{\xi(\lambda)}{2}}\left(\dfrac{2}{\mu^2 A}\right),
\end{align*}
so that subsequently from~\eqref{eq:eigfun-gen-form2} and the fact that $q(x)\propto\mathfrak{m}(x)\,\varphi(x,\lambda)$ where $\mathfrak{m}(x)$ is given by~\eqref{eq:speed-measure} it follows that the function
\begin{align*}
q(x)
&=
\dfrac{\dfrac{1}{x}\,e^{-\tfrac{1}{\mu^2 x}}\,W_{1,\tfrac{\xi(\lambda)}{2}}\left(\dfrac{2}{\mu^2 x}\right)}{e^{-\tfrac{1}{\mu^2 A}}\,W_{0,\tfrac{\xi(\lambda)}{2}}\left(\dfrac{2}{\mu^2 A}\right)}\,\indicator{x\in[0,A]}
\end{align*}
satisfies not only the entrance boundary condition~\eqref{eq:bnd-cond-0} but also the normalization constraint~\eqref{eq:norm-constraint}.

It remains to find $\lambda$, i.e., the dominant eigenvalue of the operator $\mathscr{D}$ given by~\eqref{eq:Doperator-def}. Formally, the problem is to find the largest root $\lambda\le 0$ of the equation $\varphi(A,\lambda)=0$ for a given $A>0$. More concretely, in view of~\eqref{eq:bnd-cond-A},~\eqref{eq:psi-bdd-conds} and~\eqref{eq:eigfun-gen-form2}, the equation to recover $\lambda$ from is
\begin{align}\label{eq:eigval-eqn}
W_{1,\tfrac{\xi(\lambda)}{2}}\left(\dfrac{2}{\mu^2 A}\right)
&=
0,
\end{align}
where $A>0$ is fixed and $\xi(\lambda)$ is as in~\eqref{eq:xi-def}. This equation was previously analyzed by~\cite{Polunchenko:SA2016}, who, in particular, {\em explicitly} showed it to be consistent for any $A>0$, just as predicted by the general Sturm--Liouville theory.

We have now established our first result, formally stated in the following theorem.
\begin{theorem}
For any fixed detection threshold $A>0$, the GSR statistic's quasi-stationary density $q_A(x)$ is given by
\begin{align}\label{eq:QST-pdf-answer}
q_A(x)
&=
\dfrac{\dfrac{1}{x}\,e^{-\tfrac{1}{\mu^2 x}}\,W_{1,\tfrac{\xi(\lambda)}{2}}\left(\dfrac{2}{\mu^2 x}\right)}{e^{-\tfrac{1}{\mu^2 A}}\,W_{0,\tfrac{\xi(\lambda)}{2}}\left(\dfrac{2}{\mu^2 A}\right)}\,\indicator{x\in[0,A]},
\end{align}
where $\lambda$ is the largest (nonpositive) solution of the (surely consistent) equation
\begin{align*}
W_{1,\tfrac{\xi(\lambda)}{2}}\left(\frac{2}{\mu^2 A}\right)
&=
0
\end{align*}
with $\xi(\lambda)$ as in~\eqref{eq:xi-def}. The respective quasi-stationary cdf $Q_A(x)$ is given by
\begin{empheq}[%
    left={%
        Q_A(x)=%
    \empheqlbrace}]{align}\nonumber
&1,\;\text{for $x\ge A$;}\label{eq:QST-cdf-answer}\\
&\dfrac{e^{-\tfrac{1}{\mu^2 x}}\,W_{0,\tfrac{\xi(\lambda)}{2}}\left(\dfrac{2}{\mu^2 x}\right)}{e^{-\tfrac{1}{\mu^2 A}}\,W_{0,\tfrac{\xi(\lambda)}{2}}\left(\dfrac{2}{\mu^2 A}\right)},\;\text{for $x\in[0,A)$;}\\\nonumber
&0,\;\text{otherwise}.
\end{empheq}
\end{theorem}
\begin{proof}
We only need to show~\eqref{eq:QST-cdf-answer}. To that end, it suffices to integrate~\eqref{eq:QST-pdf-answer} with respect to $x$ and evaluate the integral explicitly using the definite integral identity~\eqref{eq:Whit-int-formula}.
\end{proof}

As complicated as the obtained formulae~\eqref{eq:QST-pdf-answer} and~\eqref{eq:QST-cdf-answer} may seem, both are all perfectly amenable to numerical evaluation, especially using {\em Mathematica}, which, as was previously demonstrated, e.g., by~\cite{Polunchenko:SA2016} and by~\cite{Linetsky:OR2004,Linetsky:HandbookChapter2007}, can efficiently handle a variety of special functions, including the Whittaker $W$ function. The only ``delicate'' part is the computation of the dominant eigenvalue $\lambda$, which, as can be seen from~\eqref{eq:QST-pdf-answer} and~\eqref{eq:QST-cdf-answer}, is required for both the quasi-stationary pdf $q_A(x)$ as well as the corresponding cdf $Q_A(x)$. To that end, the equation~\eqref{eq:eigval-eqn} that determines $\lambda$ can be easily treated numerically in {\em Mathematica} as well: one can get $\lambda$ computed to within any desirable accuracy. By way of illustration of this point, we were able to get {\em Mathematica} compute $\lambda$ to within as much as 400 (four hundred!) decimal places in about a minute, for each given $A>0$, using an average laptop. Since 400 decimal places is extremely accurate, the corresponding numerically computed values of $\lambda$ can be considered ``exact'', and we shall use this terminology throughout the remainder of the paper. Figures~\ref{fig:qA+QA_A20} were obtained with the aid of our {\em Mathematica} script, and depict the quasi-stationary pdf $q_A(x)$, shown in Figure~\ref{fig:qA_A20}, and the corresponding cdf $Q_A(x)$, shown in Figure~\ref{fig:QA_A20}, as functions of $x\in[0,A]$, assuming $A=20$ and $\mu$ is either $1/2$, $1$, or $3/2$. Note that, in either case, the pdf appears to be unimodal. As a matter of fact, this property of the quasi-stationary distribution will be formally established in Theorem~\ref{thm:qA-unimodal} below.
\begin{figure}[h!]
    \centering
    \begin{subfigure}{0.48\textwidth}
        \centering
        \includegraphics[width=\linewidth]{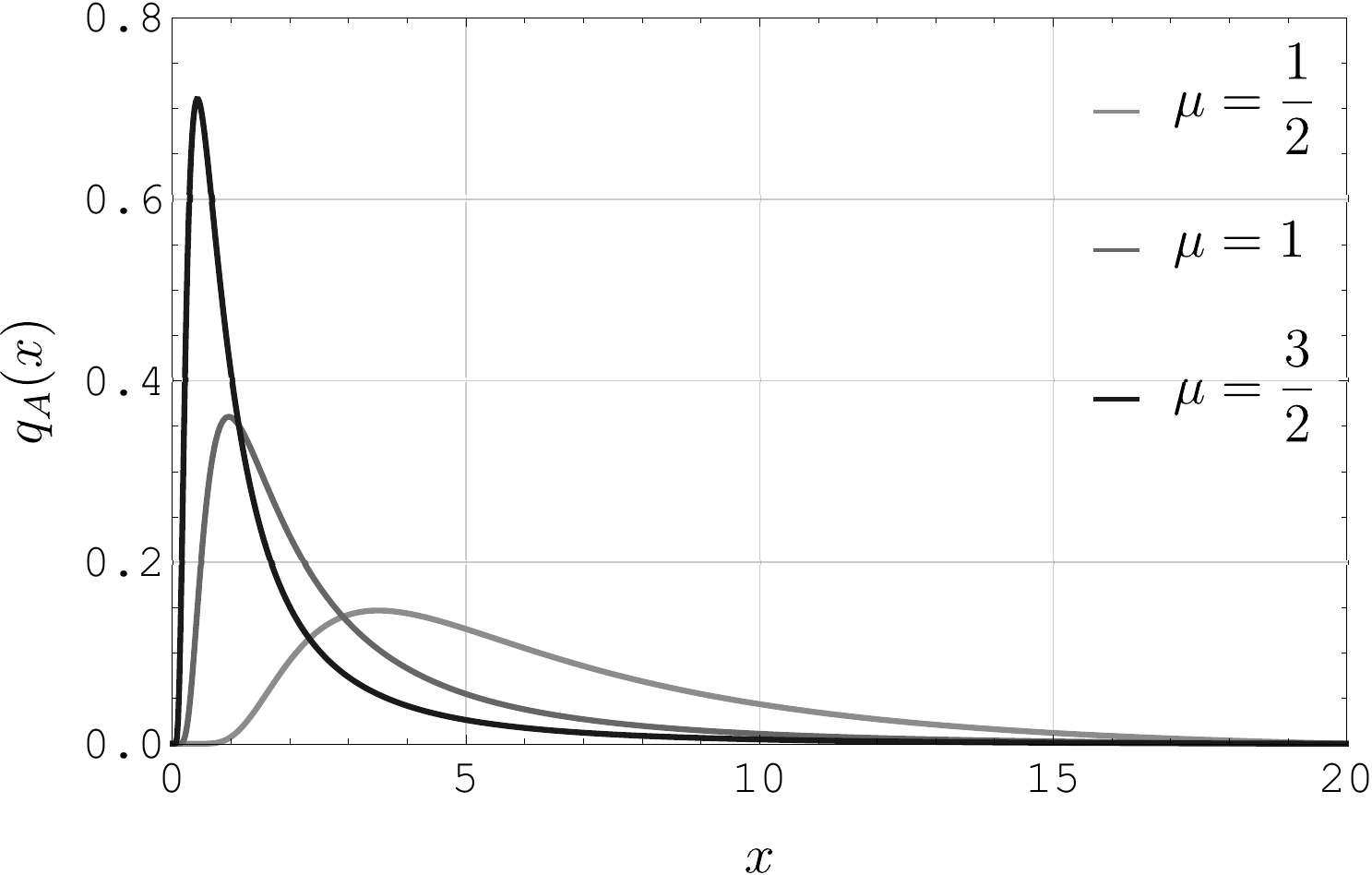}
        \caption{$q_A(x)$.}
        \label{fig:qA_A20}
    \end{subfigure}
    \hspace*{\fill}
    \begin{subfigure}{0.48\textwidth}
        \centering
        \includegraphics[width=\linewidth]{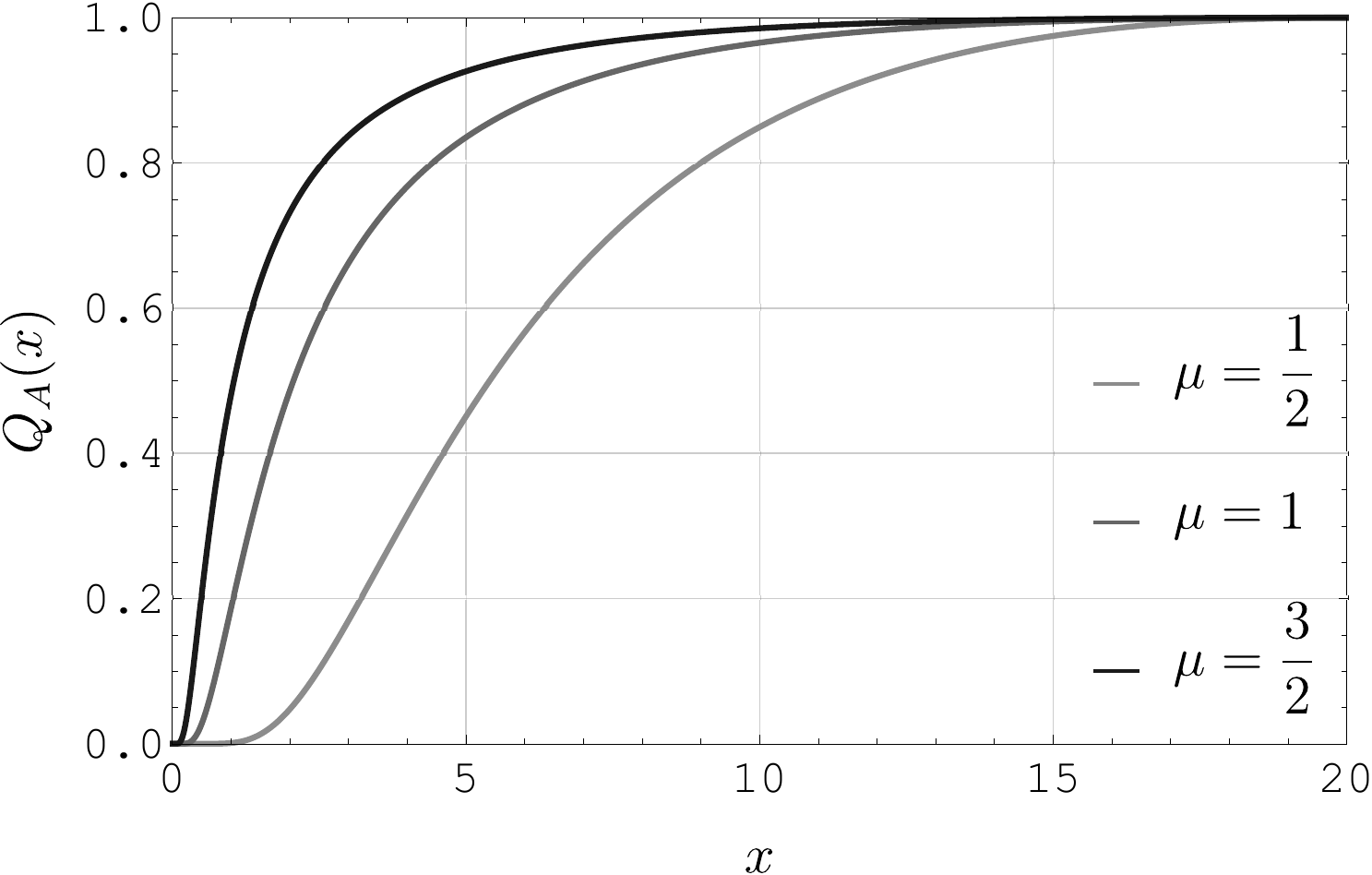}
        \caption{$Q_A(x)$.}
        \label{fig:QA_A20}
    \end{subfigure}
    \caption{Quasi-stationary distribution's pdf $q_A(x)$ and cdf $Q_A(x)$ for $A=20$ and $\mu=\{1/2,1,3/2\}$.}
    \label{fig:qA+QA_A20}
\end{figure}

On the other hand, the ``explicitness'' of formulae~\eqref{eq:QST-pdf-answer} and~\eqref{eq:QST-cdf-answer} allows to study the quasi-stationary distribution in an {\em analytic} manner, as we demonstrate in a series of results established next.

Let us begin with the following auxiliary result.
\begin{lemma}\label{lem:QST-addn-properties}
For any fixed $A>0$ it is true that:
\begin{enumerate}[\itshape(a)]
    \setlength\itemsep{6mm}
    \item $\lim_{x\to0+}q_A(x)=0$; \label{lem:QST-addn-properties:part-A}
    \item $\lim_{x\to0+}\left[\dfrac{\partial}{\partial x}q_A(x)\right]=0$; \label{lem:QST-addn-properties:part-B}
    \item $A^2\dfrac{\mu^2}{2}\left.\left[\dfrac{\partial}{\partial x}q_A(x)\right]\right|_{x=A}=\lambda$. \label{lem:QST-addn-properties:part-C}
\end{enumerate}
\end{lemma}
\begin{proof}
Part~(\ref{lem:QST-addn-properties:part-A}) can be easily deduced from~\eqref{eq:QST-pdf-answer} and
\begin{align*}
W_{1,b}(u)
&=
u\,e^{-\tfrac{u}{2}}\left[1+\mathcal{O}\left(\dfrac{1}{u}\right)\right]\;\text{as}\;u\to+\infty,\;\text{for any $b\in\mathbb{C}$},
\end{align*}
which is a special case of the more general asymptotic property of the Whittaker $W$ function
\begin{align}\label{eq:Whit-fnc-asym-zinf}
W_{a,b}(z)
&=
z^a\,e^{-\tfrac{z}{2}}\left[1+\mathcal{O}\left(\dfrac{1}{z}\right)\right]\;\text{as}\;|z|\to+\infty,\;\text{for any $b\in\mathbb{C}$},
\end{align}
established, e.g., in~\cite[Section~16.3]{Whittaker+Watson:Book1927}.

To prove part~(\ref{lem:QST-addn-properties:part-B}) explicitly, note that due to~\cite[Formula~(2.4.18),~p.~25]{Slater:Book1960} which states that
\begin{align*}
\dfrac{\partial^n}{\partial z^n}\left[z^{a+n-1}e^{-\tfrac{z}{2}}\,W_{a,b}(z)\right]
&=
(-1)^n\,z^{a-1}e^{-\tfrac{z}{2}}\,W_{a+n,b}(z),
\end{align*}
we have
\begin{align*}
\dfrac{\partial}{\partial z}\left[z\,e^{-\tfrac{z}{2}}\,W_{1,b}(z)\right]
&=
-e^{-\tfrac{z}{2}}\,W_{2,b}(z),
\end{align*}
so that by direct differentiation of~\eqref{eq:QST-pdf-answer} with respect to $x$ one can subsequently see that
\begin{align}\label{eq:qA-second-xderiv}
\dfrac{\partial}{\partial x}q_A(x)
&\propto
\dfrac{1}{x^2}\,e^{-\tfrac{1}{\mu^2 x}}\,W_{2,\tfrac{\xi}{2}}\left(\dfrac{2}{\mu^2 x}\right),
\end{align}
and this goes to zero as $x\to0+$ because
\begin{align*}
W_{2,b}(u)
&=
u^2\,e^{-\tfrac{u}{2}}\left[1+\mathcal{O}\left(\dfrac{1}{u}\right)\right]\;\text{as}\;u\to+\infty,\;\text{for any $b\in\mathbb{C}$},
\end{align*}
which is another special case of~\eqref{eq:Whit-fnc-asym-zinf}.

Finally, the assertion of part~(\ref{lem:QST-addn-properties:part-C}) will become apparent if one first integrates the master equation~\eqref{eq:master-eqn} through with respect to $x$ over $[0,A]$, and then invokes parts~(\ref{lem:QST-addn-properties:part-A}) and~(\ref{lem:QST-addn-properties:part-B}) of this lemma, the normalization constraint~\eqref{eq:norm-constraint}, and the absorbing boundary condition~\eqref{eq:bnd-cond-A} to simplify the result of the integration.
\end{proof}

Lemma~\ref{lem:QST-addn-properties} is instrumental in getting the moments of the quasi-stationary distribution.
\begin{theorem}\label{thm:QST-moments}
Let $A>0$ be fixed, and suppose that $Z$ is a $Q_A$-distributed random variable. Then the moments $M_n\triangleq\EV[Z^n]$, $n\ge0$, satisfy the recurrence
\begin{align}\label{eq:QST-moment-eqn}
M_{n}\left[\dfrac{\mu^2}{2}n(n-1)-\lambda\right]+nM_{n-1}
&=
-\lambda\,A^n,\;\;n\ge1,
\end{align}
with $M_0\equiv1$.
\end{theorem}
\begin{proof}
It suffices to multiply the master equation~\eqref{eq:master-eqn} through by $x^n$ and integrate the result with respect to $x$ over the interval $[0,A]$ using integration by parts in conjunction with Lemma~\ref{lem:QST-addn-properties}.
\end{proof}

The recurrence~\eqref{eq:QST-moment-eqn} is essentially a nonhomogeneous first-order difference equation, and although it {\em can} be solved explicitly, the solution is too cumbersome to be of significant practical value. This notwithstanding, it is a simple exercise to infer from Theorem~\ref{thm:QST-moments} that, if $Z$ is a $Q_A$-distributed random variable and $A>0$ is fixed, then $\EV[Z]\triangleq M_1=A+1/\lambda$, and
\begin{align*}
\Var[Z]
&=
M_2-M_1^2
=
-\dfrac{\mu^2(A+1/\lambda)^2+1/\lambda}{\mu^2-\lambda},
\end{align*}
where $\lambda\equiv\lambda_A$ is the largest (nonpositive) zero of equation~\eqref{eq:eigval-eqn}. It is noteworthy that the foregoing formula trivially yields the following double inequality:
\begin{align*}
-\dfrac{1}{A}-\dfrac{1+\sqrt{4\mu^2 A+1}}{2\mu^2 A^2}
\le
\lambda_A
\le
-\dfrac{1}{A}-\dfrac{1-\sqrt{4\mu^2 A+1}}{2\mu^2 A^2}\;(<0),
\end{align*}
which makes it direct to see that $\lambda_A=-1/A+\mathcal{O}(A^{-3/2})$. Although the obtained bounds allow one to estimate $\lambda_A$ for any given $A>0$, and the accuracy improves as $A$ gets higher, in the next subsection we will offer a more accurate (viz. order-three) large-$A$ approximation to $\lambda_A$.

Another interesting property of the quasi-stationary distribution is its unimodality, which we already observed in Figures~\ref{fig:qA+QA_A20} above.
\begin{theorem}\label{thm:qA-unimodal}
For any fixed detection threshold $A>0$, the GSR statistic's quasi-stationary distribution given by~\eqref{eq:QST-pdf-answer} and by~\eqref{eq:QST-cdf-answer} is unimodal.
\end{theorem}
\begin{proof}
The problem is effectively to show that for any given $A>0$ the equation
\begin{align}\label{eq:qA-deriv-zero}
\dfrac{d}{dx}q(x)
&=
0
\;\;\text{with}\;\;
x\in(0,A)
\end{align}
has only one solution, say $\tilde{x}\triangleq\tilde{x}_{A}$, and that this solution is such that the concavity condition
\begin{align}\label{eq:qA-concav-cond}
\left.\left[\dfrac{d^2}{dx^2}q(x)\right]\right|_{x=\tilde{x}}
&<
0
\end{align}
is satisfied. We point out that by part~(\ref{lem:QST-addn-properties:part-B}) of Lemma~\ref{lem:QST-addn-properties} the point $x=0$ is in the solution set of equation~\eqref{eq:qA-deriv-zero}. However, it cannot be a maximum of $q(x)$, for $\lim_{x\to0+}q(x)=0$, as was asserted by part~(\ref{lem:QST-addn-properties:part-A}) of Lemma~\ref{lem:QST-addn-properties}. Hence, the interval in~\eqref{eq:qA-deriv-zero} does not include the left end-point. The right end-point is not included because $q(A)=0$ by design.

Let us first show that the concavity condition~\eqref{eq:qA-concav-cond} is satisfied automatically for {\em any} solution $\tilde{x}\in(0,A)$ of equation~\eqref{eq:qA-deriv-zero}. To that end, note that the master equation~\eqref{eq:master-eqn} considered at any such $\tilde{x}$ can be rewritten in the form
\begin{align*}
\mu^2\tilde{x}^2\left.\left[\dfrac{d^2}{d x^2} q(x)\right]\right|_{x=\tilde{x}}
&=
2\,(\lambda-\mu^2)\,q(\tilde{x}),
\end{align*}
whence, in view of Lemmas~\ref{lem:qA-nozeros} and~\ref{lem:Dop-eigval-neg}, the desired conclusion readily follows. Put another way, whatever extrema $\tilde{x}$ the function $q(x)$ may have inside the interval $(0,A)$, they must all be the function's maxima.

It remains to show that $q(x)$ has exactly one extremum inside the interval $(0,A)$. To that end, it can be gleaned from~\eqref{eq:qA-second-xderiv} that the extrema of the function $q(x)$ are the solutions of the equation
\begin{align*}
W_{2,\tfrac{\xi(\lambda)}{2}}\left(\dfrac{2}{\mu^2 x}\right)
&=
0,
\;\;\text{such that}\;\; x\in(0,A),
\end{align*}
where $\xi(\lambda)$ is given by~\eqref{eq:xi-def}, and $\lambda\leq0$ is the largest solution of equation~\eqref{eq:eigval-eqn}. At a slightly more abstract level, the question effectively is whether the equation $W_{2,b}(z)=0$ with $b$ fixed has exactly one (finite) solution, say $\tilde{z}$, that is to the right of the largest (but finite) positive zero of the function $W_{1,b}(z)$. This equation can be elegantly answered in the affirmative with the aid of the work of Milne~\citeyearpar{Milne:PEMS1914}; see also~\cite[Theorem~V,~p.~504]{Sharma:JS1938}. Specifically, Milne~\citeyearpar{Milne:PEMS1914} showed that if $z_1$ and $z_2$ are any two (finite) successive zeros of the function $W_{k,b}(z)$ with the indices $k$ and $b$ fixed, then the function $W_{k+1,b}(z)$ has exactly one zero located between $z_1$ and $z_2$. With regard to the assumptions that $k$ and $b$ are to satisfy in order for the claim to hold true, while they are not explicitly stated in the paper, it is clear from Milne's~\citeyearpar{Milne:PEMS1914} proof that it certainly ``goes through'' if $k$ is purely real and $b$ is either purely real or purely imaginary. This is sufficient for our purposes, for in our case $k=1$ and $b$ is either purely real or purely imaginary, and the two zeros of $W_{1,b}(z)$ are $z_1=2/(\mu^2 A)$ and $z_2=\infty$. Hence, the equation $W_{2,b}(z)=0$ must have exactly one (finite) solution in between, and the proof is complete.
\end{proof}

We remark that the stationary distribution $h(x)\triangleq\lim_{A\to+\infty} q_A(x)=\mathfrak{m}(x)$ supported on $x\in[0,+\infty)$, and where $\mathfrak{m}(x)$ is given by~\eqref{eq:speed-measure}, is also unimodal: the mode is at $x^{*}=1/\mu^2$. The mode of the quasi-stationary distribution is more difficult to obtain. However, if $A$ is sufficiently large, then the asymptotic approximation of $q_A(x)$ offered in the next subsection can be used to get the mode approximately.

\subsection{Asymptotic Analysis}

We begin with two auxiliary results.
\begin{lemma}
Let $A>0$ be fixed, and suppose that $\{\lambda,\varphi(x,\lambda)\}$ is any eigenvalue-eigenfunction pair of the operator $\mathscr{G}$ given by~\eqref{eq:Goperator-def}; recall that $\varphi(x,\lambda)$ is given by~\eqref{eq:eigfun-gen-form2}. Then
\begin{align}\label{eq:eigfun-norm}
\|\varphi(\cdot,\lambda)\|^2
&=
C^2\,\Biggl[\dfrac{\partial}{\partial\lambda}W_{1,\tfrac{\xi(\lambda)}{2}}\left(\dfrac{2}{\mu^2 A}\right)\Biggr]\left.\Biggl[\dfrac{\partial}{\partial u} W_{1,\tfrac{\xi(\lambda)}{2}}(u)\Biggr]\right|_{u=\tfrac{2}{\mu^2 A}},
\end{align}
where constant $C\neq0$ is the same constant as in the right-hand side of~\eqref{eq:eigfun-gen-form2}, and $\|\varphi(\cdot,\lambda)\|^2$ and $\xi(\lambda)$ are given by~\eqref{eq:eigfun-norm-def} and by~\eqref{eq:xi-def}, respectively.
\end{lemma}
\begin{proof}
See~\cite{Polunchenko:SA2016}.
\end{proof}

The foregoing lemma enables us to establish the following important monotonicity property of the spectrum of the operator $\mathscr{D}$ given by~\eqref{eq:Doperator-def} with respect to the detection threshold $A$.
\begin{lemma}\label{lem:Dop-eigval-monotonicity}
The dominant eigenvalue $\lambda\equiv\lambda_A$ of the operator $\mathscr{D}$ is a monotonically increasing function of $A>0$. More concretely,
\begin{align*}
\dfrac{d}{d A}\lambda_A
>0
\;\text{for all}\;
A>0.
\end{align*}
\end{lemma}
\begin{proof}
Let us show that the desired result is actually valid for {\em any} eigenvalue of the operator $\mathscr{D}$. Specifically, suppose temporarily that $\lambda$ and $\varphi(x,\lambda)$ are an {\em arbitrary} eigenvalue-eigenfunction pair of $\mathscr{G}$. By implicit differentiation it is then direct to see from~\eqref{eq:eigval-eqn} that
\begin{align*}
\dfrac{d}{d A}\lambda_A
&=
-\dfrac{\dfrac{\partial}{\partial A} W_{1,\tfrac{\xi(\lambda)}{2}}\left(\dfrac{2}{\mu^2 A}\right)}{\dfrac{\partial}{\partial\lambda} W_{1,\tfrac{\xi(\lambda)}{2}}\left(\dfrac{2}{\mu^2 A}\right)}
=
\dfrac{\dfrac{2}{\mu^2 A^2}\left.\Biggl[\dfrac{\partial}{\partial u} W_{1,\tfrac{\xi(\lambda)}{2}}(u)\Biggr]\right|_{u=\tfrac{2}{\mu^2 A}}}{\dfrac{\partial}{\partial \lambda} W_{1,\tfrac{\xi(\lambda)}{2}}\left(\dfrac{2}{\mu^2 A}\right)},
\end{align*}
and because from~\eqref{eq:eigfun-norm} we also have
\begin{align*}
\dfrac{\partial}{\partial\lambda}W_{1,\tfrac{\xi(\lambda)}{2}}\left(\dfrac{2}{\mu^2 A}\right)
&=
\|\varphi(\cdot,\lambda)\|^2\left/
\left(C^2\left.\Biggl[\frac{\partial}{\partial u} W_{1,\tfrac{\xi(\lambda)}{2}}(u)\Biggr]\right|_{u=\tfrac{2}{\mu^2 A}}\right)\right.,
\end{align*}
one can readily conclude that
\begin{align*}
\dfrac{d}{d A}\lambda(A)
&=
\dfrac{2}{\mu^2 A^2}\left(C\left.\Biggl[\dfrac{\partial}{\partial u} W_{1,\tfrac{\xi(\lambda)}{2}}(u)\Biggr]\right|_{u=\tfrac{2}{\mu^2 A}}\right)^2\left/\;\|\varphi(\cdot,\lambda)\|^2\right.,
\end{align*}
which gives the desired (more general) result.
\end{proof}

Lemma~\ref{lem:Dop-eigval-monotonicity} explicitly shows that to say ``large $A$'' is the same as to say ``small $\lambda\equiv\lambda_A$''. This means, in particular, that $\lim_{A\to+\infty}\lambda_A=0$, which makes perfect sense. The large-$A$ behavior of the quasi-stationary distribution can therefore be extracted from the Taylor series of the function $W_{1,\tfrac{\xi(\lambda)}{2}}(x)$ present in the numerator of the fraction in the right-hand side of~\eqref{eq:QST-pdf-answer} expanded with respect to $\lambda$ around zero, assuming that $x\in[0,A]$ is fixed. Specifically, noting from~\eqref{eq:xi-def} that $\xi(0)=1$, the problem is to obtain the expansion
\begin{align}\label{eq:Whit-fcn-taylor-expn}
W_{1,\tfrac{\xi(\lambda)}{2}}(x)
&=
W_{1,\tfrac{1}{2}}(x)
+
\sum_{k=1}^{n}\left.\left[\dfrac{\partial^k}{\partial\lambda^k}W_{1,\tfrac{\xi(\lambda)}{2}}(x)\right]\right|_{\lambda=0}\dfrac{\lambda^k}{k!}
+
\mathcal{O}(\lambda^{n+1}),
\end{align}
up to a desired order $n\ge1$; while the order $n$ can, in principle, be 0, shrinking the foregoing Taylor series to just the first term would not yield a particularly accurate approximation. With regard to the denominator of the fraction in the right-hand of~\eqref{eq:QST-pdf-answer}, it can be safely forgotten about, because
\begin{align*}
\lim_{A\to+\infty}e^{-\tfrac{1}{\mu^2 A}}\,W_{0,\tfrac{\xi(\lambda_A)}{2}}\left(\dfrac{2}{\mu^2 A}\right)
&=
1,
\end{align*}
which follows from the fact that $\lim_{x\to0+}x^x=1$ and the Whittaker $W$ function's small-argument asymptotics
\begin{align*}
W_{0,b}(u)
&\sim
\dfrac{\Gamma(2b)}{\Gamma(b+1/2)} u^{\tfrac{1}{2}-b}\,e^{-\tfrac{u}{2}}\;\text{as}\;u\to0+,\;\text{for any $b\in\mathbb{C}$},
\end{align*}
which, in turn, is a special case of
\begin{align*}
W_{a,b}(u)
&\sim
\dfrac{\Gamma(2b)}{\Gamma(b-a+1/2)} u^{\tfrac{1}{2}-b}\,e^{-\tfrac{u}{2}}\;\text{as}\;u\to0+,\;\text{for any $a,b\in\mathbb{C}$}.
\end{align*}

The first term of the sought series~\eqref{eq:Whit-fcn-taylor-expn} can be worked out from~\cite[Identity~(28a),~p.~23]{Buchholz:Book1969}, according to which
\begin{align*}
W_{a,a-\tfrac{1}{2}}(z)
&=
z^{a}e^{-\tfrac{z}{2}},
\end{align*}
so that
\begin{align*}
W_{1,\tfrac{1}{2}}(z)
&=
z\,e^{-\tfrac{z}{2}},
\end{align*}
and we can proceed to finding the higher-order terms. To that end, the task requires the evaluation of the partial derivatives
\begin{align*}
\left.\left[\dfrac{\partial^{k}}{\partial b^{k}}W_{1,b}(x)\right]\right|_{b=\tfrac{1}{2}},
\end{align*}
for a fixed $x\in[0,A]$ and for all $k$'s starting from 1 and up to the desired order $n\ge1$ of the expansion~\eqref{eq:Whit-fcn-taylor-expn}. We shall now set $n=3$ and find the necessary three derivatives, so as to then obtain the corresponding {\em third-order} expansion. This will yield a sufficiently accurate approximation of the quasi-stationary distribution, and augmenting the expansion with higher-order terms---though {\em is} possible---would be superfluous.

To proceed, let us recall two special functions that will be needed below. The first function we will need is the exponential integral
\begin{empheq}[%
    left={%
        \Ei(x)\triangleq%
    \empheqlbrace}]{align*}
&-\displaystyle\int_{-x}^\infty\dfrac{e^{-y}}{y}\,dy,\;\text{for $x<0$;}\\
&-\lim_{\varepsilon\to+0}\left[\int_{-x}^{-\varepsilon}\dfrac{e^{-y}}{y}\,dy+\int_{\varepsilon}^{\infty}\dfrac{e^{-y}}{y}\,dy\right],\;\text{for $x>0$},
\end{empheq}
whose basic properties are summarized, e.g., in~\cite[Chapter~5]{Abramowitz+Stegun:Handbook1964}. More specifically, we will need the function $\E1(x)\triangleq-\Ei(-x)$, but restricted to positive values of the argument, so that
\begin{align}\label{eq:E1-func-def}
\E1(x)
&\triangleq
\int_{x}^\infty\dfrac{e^{-y}}{y}\,dy
=
x e^{-x}\int_{0}^{+\infty} e^{-xy}\log(1+y)\,dy,\; x>0,
\end{align}
where the second equality is due to~\cite[Integral~4.337.1,~p.~572]{Gradshteyn+Ryzhik:Book2007} which is
\begin{align*}
\int_{0}^{+\infty} e^{-ay}\log(b+y)\,dy
&=
\dfrac{1}{a}\big[\log b-e^{ab}\Ei(-ba)\big],
\;
\abs{\arg b}<\pi,\; \Re(a)>0.
\end{align*}

The second function we will need is the Meijer $G$-function, introduced in the seminal work of Meijer~\citeyearpar{Meijer:NAW1936}. The Meijer $G$-function is defined as the Mellin-Barnes integral
\begin{align}\label{eq:MeijerG-func-def}
\MeijerG*{m}{n}{p}{q}{a_1,\ldots,a_p}{b_1,\ldots,b_q}{z}
&\triangleq
\dfrac{1}{2\pi\imath}\int_{\mathcal{C}}\dfrac{\prod_{k=1}^m \Gamma(b_k-s) \prod_{j=1}^n \Gamma(1-a_j+s)}{\prod_{k=m+1}^q \Gamma(1-b_k+s) \prod_{j=n+1}^p \Gamma(a_j-s)}\,z^s ds,
\end{align}
where $\imath$ denotes the imaginary unit, i.e., $\imath\triangleq\sqrt{-1}$, the integers $m$, $n$, $p$, and $q$ are such that $0\le m\le q$ and $0\le n\le p$, and the contour $\mathcal{C}$ is closed in an appropriate way to ensure the convergence of the integral. It is also required that no $a_j-b_k$ be an integer. The function is a very general function, and includes, as special cases, not only all elementary functions, but a number of special functions as well. An extensive list of special cases of the Meijer $G$-function can be found, e.g., in the classical special functions handbook of~\cite{Prudnikov+etal:Book1990}, which also includes a summary of the function's basic properties. We will need the following particular case of the Meijer $G$-function:
\begin{align}\label{eq:MeijerG-func-spcl-case}
\MeijerG*{3}{1}{2}{3}{0,1}{0,0,0}{x}
&=
\int_{x}^{+\infty}e^{y}\E1(y)\,\dfrac{dy}{y}
=
\int_{0}^{+\infty} e^{-xy}\log(1+y)\,\dfrac{dy}{y},\; x>0;
\end{align}
see Appendix~\ref{sec:apndx:MeijerG-func-spcl-case-proof} for a proof.

We are now in a position to formulate the main result that will subsequently enable us to obtain the expansion~\eqref{eq:Whit-fcn-taylor-expn} explicitly, up to the third order.
\begin{lemma}\label{lem:Whit-fcn-deriv}
For any fixed $x\ge0$ the following identities hold true:
\begin{enumerate}[\itshape(a)]
    \item $\left.\left[\dfrac{\partial}{\partial b}W_{1,b}(x)\right]\right|_{b=\tfrac{1}{2}}=e^{-\tfrac{x}{2}}$; \label{lem:Whit-fcn-deriv:part-A}
    \item $\left.\left[\dfrac{\partial^2}{\partial b^2}W_{1,b}(x)\right]\right|_{b=\tfrac{1}{2}}=2\,e^{-\tfrac{x}{2}}\left\{e^{x}\E1(x)+x\,\MeijerG*{3}{1}{2}{3}{0,1}{0,0,0}{x}\right\}$; \label{lem:Whit-fcn-deriv:part-B}
    \item $\left.\left[\dfrac{\partial^3}{\partial b^3}W_{1,b}(x)\right]\right|_{b=\tfrac{1}{2}}=6\,e^{-\tfrac{x}{2}}\MeijerG*{3}{1}{2}{3}{0,1}{0,0,0}{x}$. \label{lem:Whit-fcn-deriv:part-C}
\end{enumerate}
\end{lemma}

Our proof of this lemma is offered in Appendix~\ref{sec:apndx:Whit-fcn-deriv-proof}, and it is based on repetitive differentiation of the Mellin integral transform (see, e.g.,~\citealt[pp.~1--10]{Oberhettinger:Book1974}) of the function $f(x)\triangleq e^{-\tfrac{x}{2}}\,W_{1,b}(x)$, $x\ge0$, with respect to the second index $b$. As an aside, we note that, apparently, the identities established in Lemma~\ref{lem:Whit-fcn-deriv} have not yet been derived in the theory of special functions (at least we were unable to find an existing reference for any one of them).

It is now straightforward to obtain the sought third-order expansion. Specifically, Lemma~\ref{lem:Whit-fcn-deriv} and the basic differentiation chain rule together lead to the following result.
\begin{theorem}\label{thm:Whit-func-expn}
For any $x\ge0$ it holds true that
\begin{align}\label{eq:Whit-func-expn}
\begin{split}
W_{1,\frac{\xi(\lambda)}{2}}&\left(\dfrac{2}{\mu^2 x}\right)
=
\dfrac{2}{\mu^2}\,e^{-\tfrac{1}{\mu^2 x}}\Biggl\{\dfrac{1}{x}+\lambda+\dfrac{2}{\mu^2}L\left(\dfrac{2}{\mu^2 x}\right)\lambda^2+\\
&\quad+
\left(\dfrac{2}{\mu^2}\right)^2\left[\MeijerG*{3}{1}{2}{3}{0,1}{0,0,0}{\dfrac{2}{\mu^2 x}}-2L\left(\dfrac{2}{\mu^2 x}\right)\right]\lambda^3\Biggr\}
+\mathcal{O}(\lambda^4),
\end{split}
\end{align}
where $\xi(\lambda)$ is as in~\eqref{eq:xi-def}, and
\begin{align}\label{eq:LwrBnd-def}
L(x)
&\triangleq
e^{x}\E1(x)-1+x\,\MeijerG*{3}{1}{2}{3}{0,1}{0,0,0}{x}.
\end{align}
\end{theorem}

As the first application of the foregoing theorem, let us obtain an order-one, order-two, and order-three approximations of the dominant eigenvalue $\lambda$. To that end, it suffices to apply~\eqref{eq:Whit-func-expn} to approximate the Whittaker $W$ function sitting in the left-hand side of equation~\eqref{eq:eigval-eqn}. Specifically, in view of~\eqref{eq:Whit-func-expn}, approximated up to the first order, the equation is simply $1/A+\lambda=0$, so that
\begin{align*}
\lambda^{*}
&\equiv
\lambda_{A}^{*}
\triangleq
-\dfrac{1}{A}
\end{align*}
is $\lambda$'s first-order approximation.

Likewise, getting $\lambda$'s second-order approximation is a matter of solving the quadratic equation:
\begin{align*}
\dfrac{2}{\mu^2}\,L\left(\dfrac{2}{\mu^2 A}\right)\lambda^2+\lambda+\dfrac{1}{A}
&=
0,
\end{align*}
which may or may not have {\em real} solutions, depending on whether $A$ is large enough. Assuming it is, the quadratic equation has two distinct real roots, both negative, and the one closest to zero should be used as $\lambda$'s second-order approximation. Written explicitly, the second-order approximation is as follows:
\begin{align*}
\lambda^{**}
&
\equiv
\lambda_{A}^{**}
\triangleq
-\dfrac{\mu^2}{4}\,\dfrac{1-\sqrt{1-\dfrac{8}{\mu^2 A}\,L\left(\dfrac{2}{\mu^2 A}\right)}}{L\left(\dfrac{2}{\mu^2 A}\right)},
\end{align*}
and we reiterate that it requires the detection threshold $A>0$ to be sufficiently large.

Finally, the third-order approximation is determined by the cubic equation:
\begin{align*}
\dfrac{1}{A}+\lambda+\dfrac{2}{\mu^2}\,L\left(\dfrac{2}{\mu^2 A}\right)\lambda^2+
\left(\dfrac{2}{\mu^2}\right)^2\left[\MeijerG*{3}{1}{2}{3}{0,1}{0,0,0}{\dfrac{2}{\mu^2 A}}-2\,L\left(\dfrac{2}{\mu^2 A}\right)\right]\lambda^3
&=
0,
\end{align*}
which, again for sufficiently large $A$'s, has exactly one real root, because the coefficient in front of $\lambda^3$ is negative. It is that single real solution that, when exists, should be used as $\lambda$'s third-order approximation, i.e., as $\lambda^{***}\equiv\lambda_A^{***}$. While it is possible to express $\lambda^{***}$ explicitly, the formula is simply too ``bulky'' to state, and for this reason only we shall not present it.

The quality of each of the three approximations of $\lambda$ can be judged from Table~\ref{tab:lambda-approx}, which reports the ``exact'' (i.e., computed to within 400 decimal places of accuracy) value of $\lambda$, and the corresponding three approximations thereof for a handful of values of $A>0$, and assuming $\mu=1$. Specifically, the table provides only the first 12 decimal places of $-\lambda$, $-\lambda^{*}$, $-\lambda^{**}$, and of $-\lambda^{***}$, and the negation is done sheerly for convenience. It is evident from the table that the first-order approximation $\lambda^{*}$ performs somewhat poorly, unless the detection threshold is on the order of thousands, which is considered high in practice. The second-order approximation $\lambda^{**}$ performs much better, even when the detection threshold is moderate. The third-order approximation $\lambda^{***}$ is the most accurate of the three approximations, and is fairly close to the ``exact'' value of $\lambda$ for $A$'s on the order of tens, which is rather low from a practical standpoint. However, in terms of computational convenience, the ranking of the the three approximations is the exact opposite of their accuracy ranking.
\begin{table}
    \centering
    \caption{Dominant eigenvalue $\lambda$ and its first-, second- and, third-order approximations $-\lambda^{*}$, $-\lambda^{**}$, $-\lambda^{***}$ as functions of $A$ for $\mu=1$.}
    \begin{tabular}{rrrrr}
        \hline
            $A$ & $-\lambda$ & $-\lambda^{*}$ & $-\lambda^{**}$ & $-\lambda^{***}$ \\
        \hline
            $20$ & $0.058856148622$ & $0.05$ & $0.059819055496$ & $0.058817735494$ \\
            $30$ & $0.037786534271$ & $0.033333333333$ & $0.03811217223$ & $0.03777661428$ \\
            $40$ & $0.027727324417$ & $0.025$ & $0.027880519395$ & $0.027723505394$ \\
            $50$ & $0.02186160095$ & $0.02$ & $0.021947421685$ & $0.02185977578$ \\
            $100$ & $0.010563106075$ & $0.01$ & $0.010577520296$ & $0.010562921283$ \\
            $500$ & $0.002033066472$ & $0.002$ & $0.002033295282$ & $0.002033065611$ \\
            $1,000$ & $0.0010095172$ & $0.001$ & $0.001009554734$ & $0.001009517118$ \\
            $10,000$ & $0.000100139278$ & $0.0001$ & $0.000100139359$ & $0.000100139278$ \\
        \hline
    \end{tabular}
    \label{tab:lambda-approx}
\end{table}

We are now in a position to offer an order-one, order-two, and order-three ``large-$A$'' approximations of the quasi-stationary pdf $q_A(x)$ given by~\eqref{eq:QST-pdf-answer}. Let the approximations be denoted as $q_A^{*}(x)$, $q_A^{**}(x)$, and $q_A^{***}(x)$, respectively, which is in analogy to the notation adapted for the three approximations of $\lambda$. By combining Theorem~\ref{thm:Whit-func-expn} and~\eqref{eq:QST-pdf-answer} we obtain
\begin{align*}
\begin{split}
q_A^{*}(x)
&=
\dfrac{\dfrac{2}{\mu^2 x}\,e^{-\tfrac{2}{\mu^2 x}}}{e^{-\tfrac{1}{\mu^2 A}}\,W_{0,\tfrac{\xi(\lambda^{*})}{2}}\left(\dfrac{2}{\mu^2 A}\right)}\left\{\dfrac{1}{x}+\lambda^{*}\right\}\indicator{x\in[0,A]};\\
q_A^{**}(x)
&=
\dfrac{\dfrac{2}{\mu^2 x}\,e^{-\tfrac{2}{\mu^2 x}}}{e^{-\tfrac{1}{\mu^2 A}}\,W_{0,\tfrac{\xi(\lambda^{**})}{2}}\left(\dfrac{2}{\mu^2 A}\right)}\,\left\{\dfrac{1}{x}+\lambda^{**}+\dfrac{2}{\mu^2}L\left(\dfrac{2}{\mu^2 x}\right)\big(\lambda^{**}\big)^2\right\}\indicator{x\in[0,A]};\\
q_A^{***}(x)
&=
\dfrac{\dfrac{2}{\mu^2 x}\,e^{-\tfrac{2}{\mu^2 x}}}{e^{-\tfrac{1}{\mu^2 A}}\,W_{0,\tfrac{\xi(\lambda^{***})}{2}}\left(\dfrac{2}{\mu^2 A}\right)}\times\\
&\qquad\quad\times
\Biggl\{\dfrac{1}{x}+\lambda^{***}+\dfrac{2}{\mu^2}L\left(\dfrac{2}{\mu^2 x}\right)\big(\lambda^{***}\big)^2+\\
&\qquad\qquad\quad+
\left(\dfrac{2}{\mu^2}\right)^2\left[\MeijerG*{3}{1}{2}{3}{0,1}{0,0,0}{\dfrac{2}{\mu^2 x}}-2L\left(\dfrac{2}{\mu^2 x}\right)\right]\big(\lambda^{***}\big)^3\Biggr\}\,\indicator{x\in[0,A]},
\end{split}
\end{align*}
where $L(x)$ is defined in~\eqref{eq:LwrBnd-def}, and the corresponding three approximations of $\lambda$ are computed as described above.

To get an idea as to the accuracy of the obtained approximations, let us look at Figures~\ref{fig:QST-pdf-approx-comparison}. Specifically, Figure~\ref{fig:QST-pdf-approx} shows the ``exact'' pdf $q_A(x)$ and the three approximations thereof, all as functions of $x\in[0,A]$, assuming $A=20$ and $\mu=1$. Figure~\ref{fig:QST-pdf-approx-err} shows the corresponding absolute errors, i.e., the quantities $\abs{\,q_A(x)-q_A^{*}(x)}$, $\abs{\,q_A(x)-q_A^{**}(x)}$, and $\abs{\,q_A(x)-q_A^{***}(x)}$. The detection threshold is intentionally set so low, for otherwise the three approximations would be closer to the actual pdf, and the corresponding errors would be harder to notice. Observe from the figures that the first-order approximation is noticeably off. Recalling the numbers reported earlier in Table~\ref{tab:lambda-approx}, this is a direct consequence of the threshold being set so low. However, in spite of the low threshold, the second- and third-order approximations, whose corresponding curves appear to nearly coincide in the figures, are both fairly close to the actual ``exact'' pdf, across the entire range of values of $x\in[0,A]$. While the quality of all three approximations improves as the threshold gets bigger, the second- and third-order approximations each become practically indistinguishable from the exact pdf as soon as the threshold is in the hundreds, which is about the range often used in practice.
\begin{figure}[h!]
    \centering
    \begin{subfigure}{0.48\textwidth}
        \centering
        \includegraphics[width=\linewidth]{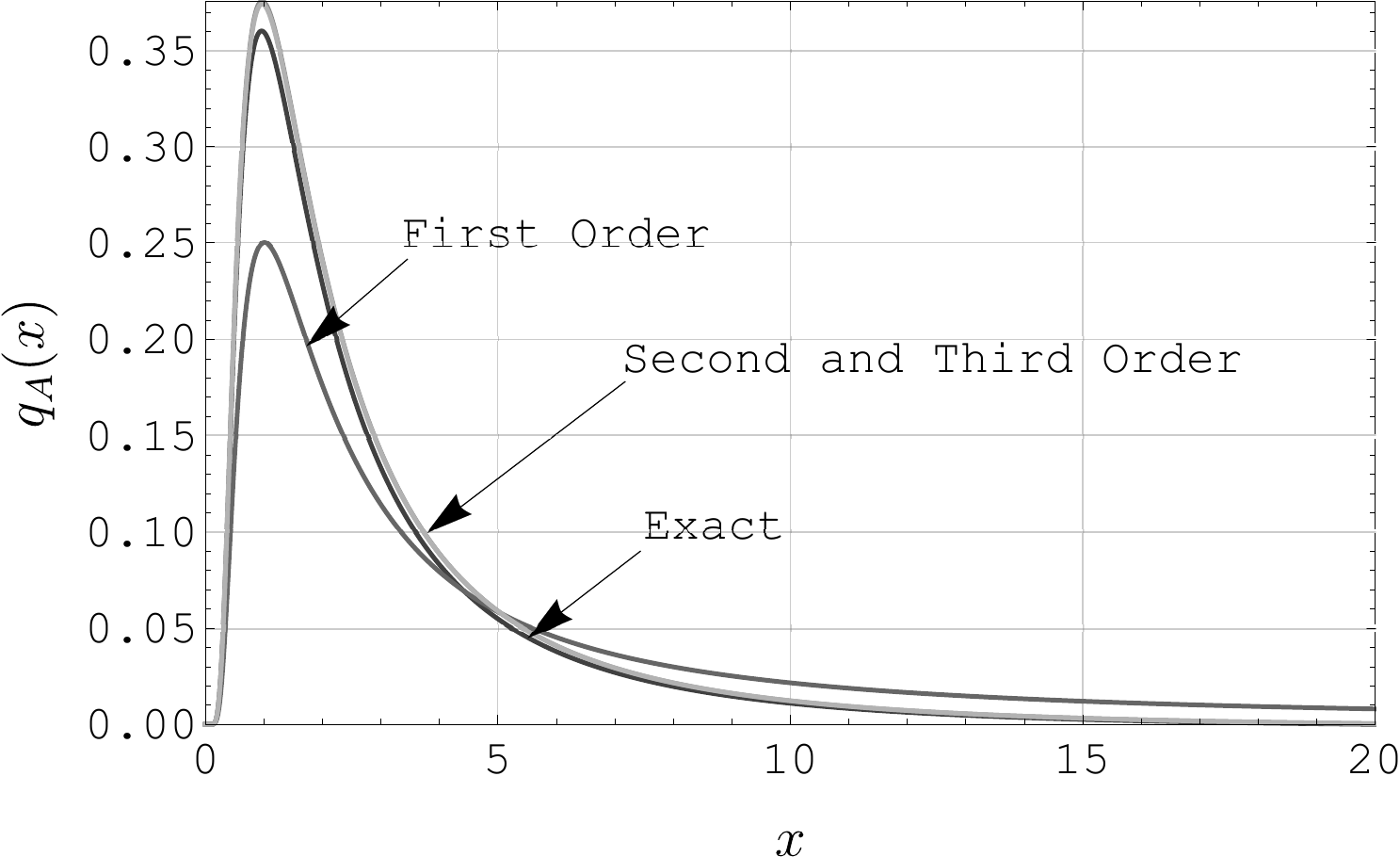}
        \caption{Exact pdf and its three approximations.}
        \label{fig:QST-pdf-approx}
    \end{subfigure}
    \hspace*{\fill}
    \begin{subfigure}{0.48\textwidth}
        \centering
        \includegraphics[width=\linewidth]{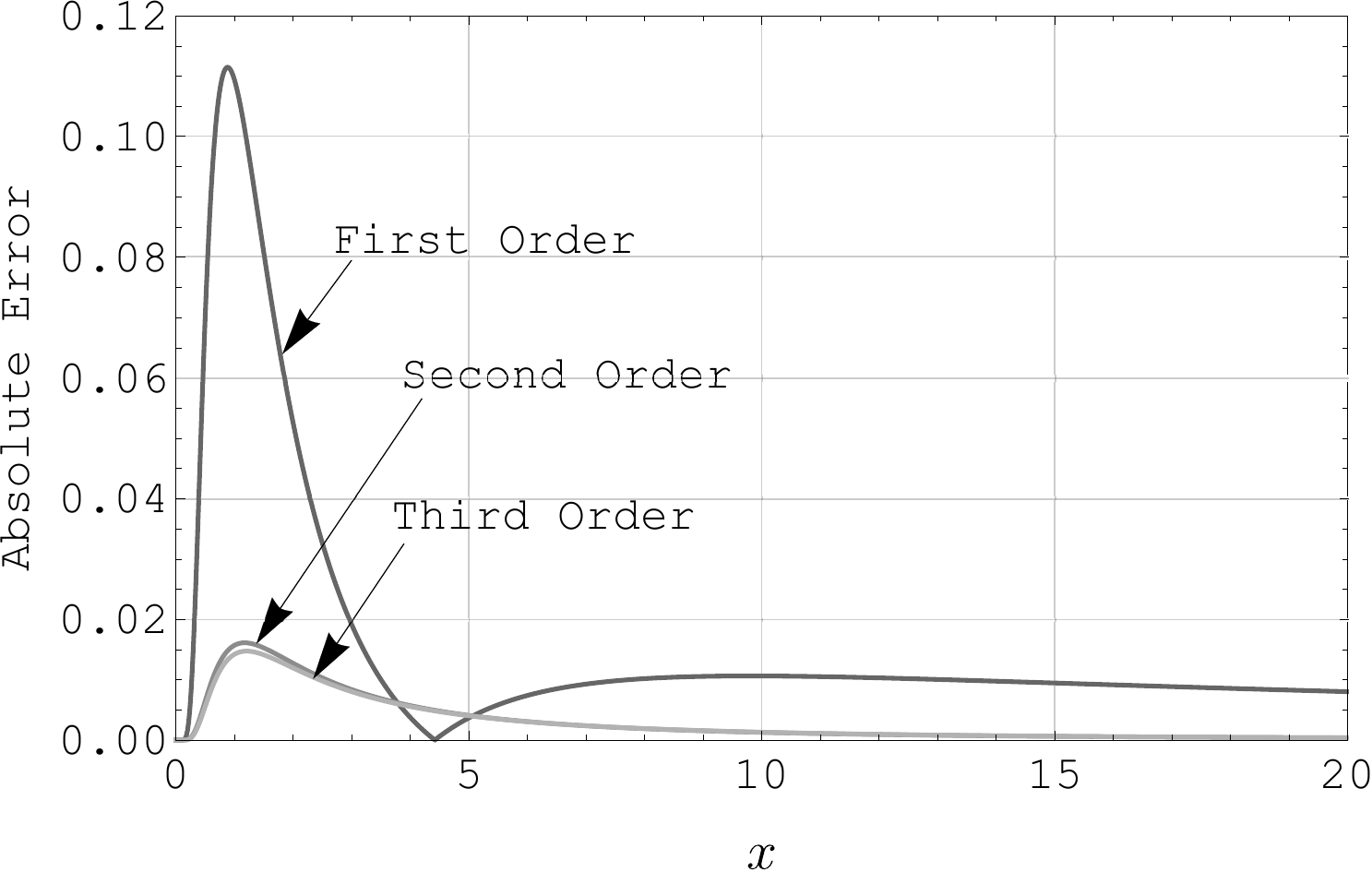}
        \caption{Corresponding absolute errors.}
        \label{fig:QST-pdf-approx-err}
    \end{subfigure}
    \caption{Quasi-stationary distribution's pdf $q_A(x)$ and its three approximations $q_A^{*}(x)$, $q_A^{**}(x)$, and $q_A^{***}(x)$ as functions of $x\in[0,A]$ for $A=20$ and $\mu=1$.}
    \label{fig:QST-pdf-approx-comparison}
\end{figure}

We conclude this subsection with a remark concerning the function $L(x)$ introduced in~\eqref{eq:LwrBnd-def}. This function plays an important role in the minimax theory of quickest change-point detection, since it is the key ingredient of the universal lowerbound (obtained, e.g., in~\citealt[Lemma~2.2]{Feinberg+Shiryaev:SD2006}) on the still-unknown optimal value of the detection delay penalty-function introduced by Pollak~\citeyearpar{Pollak:AS85}; for the discrete-time setting of the problem, the equivalent of this lowerbound was introduced by~\cite{Moustakides+etal:SS11}. Therefore, function $L(x)$ is essential in assessing the efficiency of a detection procedure, and, in particular, possibly proving that the procedure of interest is either exactly or nearly Pollak-minimax. Precisely this approach was employed by~\cite{Burnaev+etal:TPA2009} to prove that the randomized version of the (Generalized) Shiryaev--Roberts procedure is asymptotically order-two Pollak-minimax. The main apparatus used throughout our investigation, i.e., the Whittaker $W$ function and the calculus associated with it, may help improve the second-order minimaxity established by~\cite{Burnaev+etal:TPA2009}. To that end, observe that part (c) of Lemma~\ref{lem:Whit-fcn-deriv} provides a link between function $L(x)$ and the Whittaker $W$ function. Consequently, the lowerbound, which is effectively defined in~\eqref{eq:LwrBnd-def}, may be reexpressed as the second derivative of the corresponding Whittaker function (the derivative is with respect to the second index). This alternative form of the lowerbound might offer a new insight into the problem, and a further investigation of this appears to be worthwhile.

\section{Concluding Remarks}
\label{sec:conclusion}

The focus of this work was on the quasi-stationary distribution of the (Generalized) Shiryaev--Roberts (GSR) diffusion from quickest change-point detection theory. We obtained exact formulae for the distribution's cdf and pdf; the formulae are valid for any threshold that the GSR diffusion is stopped at. We also found the distribution's entire moment series, and proved the distribution to be unimodal for any threshold. More importantly, we offered an order-three large-threshold asymptotic approximation of the quasi-stationary distribution. The derivation of the approximation required establishing new identities on certain special functions. This work's findings are of direct application to quickest change-point detection, and we expect to make the transition in a subsequent paper.


\appendix

\section{Proof of Formula 3.21} 
\label{sec:apndx:MeijerG-func-spcl-case-proof}

We first note that the identity
\begin{align*}
\int_{x}^{+\infty}e^{y}\E1(y)\,\dfrac{dy}{y}
&=
\int_{0}^{+\infty} e^{-xy}\log(1+y)\,\dfrac{dy}{y},\; x>0,
\end{align*}
where $\E1(x)$ is the exponential integral defined in~\eqref{eq:E1-func-def}, has already been established, e.g., by~\cite[pp.~453--454]{Feinberg+Shiryaev:SD2006}. It therefore remains to show only that
\begin{align*}
\MeijerG*{3}{1}{2}{3}{0,1}{0,0,0}{x}
&=
\int_{0}^{+\infty} e^{-xy}\log(1+y)\,\dfrac{dy}{y},\; x>0,
\end{align*}
where $G$ denotes the Meijer $G$-function whose general definition is given by~\eqref{eq:MeijerG-func-def}. To that end, the idea is to use~\cite[Identity~8.4.6.5,~p.~537]{Prudnikov+etal:Book1990} according to which
\begin{align*}
\log(1+x)
&=
\MeijerG*{1}{2}{2}{2}{1,1}{1,0}{x},
\end{align*}
so that
\begin{align*}
\int_{0}^{+\infty} e^{-xy}\log(1+y)\,\dfrac{dy}{y}
&=
\int_{0}^{+\infty} e^{-xy}\MeijerG*{1}{2}{2}{2}{1,1}{1,0}{y}\dfrac{dy}{y},
\end{align*}
whence, in view of the definite integral
\begin{align*}
\begin{split}
\int_{0}^{+\infty} e^{-\beta x}&\MeijerG*{m}{n}{p}{q}{a_1,\ldots,a_p}{b_1,\ldots,b_q}{\alpha x}\dfrac{dx}{x^c}
=
\beta^{c-1}\MeijerG*{m}{n+1}{p+1}{q}{c,a_1,\ldots,a_p}{b_1,\ldots,b_q}{\frac{\alpha}{\beta}},\\
&\qquad\qquad\text{provided $p+q<2(m+n)$,\;$\abs{\arg\alpha}<(2m+2n-p-q)\frac{\pi}{2}$},\\
&\qquad\qquad\qquad\text{\;$\abs{\arg\beta}<\frac{\pi}{2}$, and $\Re(b_j-c)>-1$,\; $j=1,\ldots,m$},
\end{split}
\end{align*}
as given, e.g., by~\cite[Identity~7.813.1,~p.~853]{Gradshteyn+Ryzhik:Book2007}, it follows that
\begin{align*}
\int_{0}^{+\infty} e^{-xy}\log(1+y)\,\dfrac{dy}{y}
&=
\MeijerG*{1}{3}{3}{2}{1,1,1}{1,0}{\frac{1}{x}},
\end{align*}
and to end the proof, it suffices to recall one of the basic properties of the Meijer $G$-function, viz. one asserting that
\begin{align*}
\MeijerG*{m}{n}{p}{q}{a_1,\ldots,a_p}{b_1,\ldots,b_q}{z}
&=
\MeijerG*{n}{m}{q}{p}{1-b_1,\ldots,1-b_q}{1-a_1,\ldots,1-a_p}{\frac{1}{z}},
\end{align*}
as given, e.g., by~\cite[Identity~8.2.2.14,~p.~521]{Prudnikov+etal:Book1990}, and also observe that, by definition~\eqref{eq:MeijerG-func-def}, any Meijer $G$-function is symmetric with respect to $a_1,\ldots,a_n$, with respect to $a_{n+1},\ldots,a_{p}$, with respect to $b_1,\ldots,b_{m}$, and with respect to $b_{m+1},\ldots,b_{q}$---each group of parameters treated separately.

\section{Proof of Lemma 3.4} 
\label{sec:apndx:Whit-fcn-deriv-proof}

The centerpiece of our proof is the identity
\begin{align*}
\int_{0}^{+\infty}
x^{s-1} e^{-\tfrac{cx}{2}}\,W_{a,b}(cx)\,dx
&=
\dfrac{\Gamma(1/2+b+s)\,\Gamma(1/2-b+s)}{c^{s}\Gamma(1-a+s)},
\;
s>-1/2\pm\Re(b);
\end{align*}
cf., e.g.,~\cite[Identity~13.52,~p.~146]{Oberhettinger:Book1974}. This identity is the Mellin transform of the function $f(x)\triangleq e^{-\tfrac{cx}{2}}\,W_{a,b}(cx)$ with $a$, $b$, and $c$ fixed, and $x\ge0$. We are interested in the case when $c=a=1$, and when $b$ is either purely real and such that $0\le b\le 1/2$, or purely imaginary (i.e., $\Re(b)=0$). In this case the above identity reduces to
\begin{align*}
\int_{0}^{+\infty}
x^{s-1} e^{-\tfrac{x}{2}}\,W_{1,b}(x)\,dx
&=
\dfrac{\Gamma(1/2+b+s)\,\Gamma(1/2-b+s)}{\Gamma(s)},\; s>0,
\end{align*}
and the entire proof is based on successive differentiation of the foregoing with respect to $b$, followed by the evaluation of the result at $b=1/2$.

To show part~(\ref{lem:Whit-fcn-deriv:part-A}), observe first that a differentiation of the above integral identity through with respect to $b$ gives
\begin{align}\label{eq:Whit-func-Mellin-1st-deriv}
\begin{split}
\int_{0}^{+\infty}
x^{s-1} e^{-\tfrac{x}{2}}&\left[\dfrac{\partial}{\partial b}W_{1,b}(x)\right]dx
=
\dfrac{\Gamma(1/2+b+s)\,\Gamma(1/2-b+s)}{\Gamma(s)}\times\\
&\qquad\quad\times
\big[\psi_0(1/2+b+s)-\psi_0(1/2-b+s)\big],
\;
s>0,
\end{split}
\end{align}
where here and onward $\psi_0(z)$ denotes the digamma function defined as $\psi_0(z)\triangleq d\log\Gamma(z)/dz$; for the basic background on the digamma function, see, e.g.,~\cite[Section~6.3]{Abramowitz+Stegun:Handbook1964}. The substitution $b=1/2$ turns~\eqref{eq:Whit-func-Mellin-1st-deriv} into
\begin{align*}
\int_{0}^{+\infty}
x^{s-1} e^{-\tfrac{x}{2}}\left.\left[\dfrac{\partial}{\partial b}W_{1,b}(x)\right]\right|_{b=1/2}dx
&=
\Gamma(s),
\; s>0,
\end{align*}
because $\psi_0(z+1)-\psi_0(z)=1/z$, as given, e.g., by~\cite[Identity~6.3.5,~p.~258]{Abramowitz+Stegun:Handbook1964}. Finally, since $\Gamma(s)$ for $s>0$, is the Mellin transform of the function $f(x)\triangleq e^{-x}$, $x\ge0$, one can deduce that
\begin{align*}
e^{-\tfrac{x}{2}}\left.\left[\dfrac{\partial}{\partial b}W_{1,b}(x)\right]\right|_{b=1/2}
&=
e^{-x},\; x\ge0,
\end{align*}
whence the desired result.

Next, by differentiating~\eqref{eq:Whit-func-Mellin-1st-deriv} through with respect to $b$ we obtain
\begin{align}\label{eq:Whit-func-Mellin-2nd-deriv}
\begin{split}
\int_{0}^{+\infty}
x^{s-1} e^{-\tfrac{x}{2}}&\left[\dfrac{\partial^2}{\partial b^2}W_{1,b}(x)\right]dx
=
\dfrac{\Gamma(1/2+b+s)\,\Gamma(1/2-b+s)}{\Gamma(s)}\times\\
&\qquad\times
\Biggl\{\big[\psi_0(1/2+b+s)-\psi_0(1/2-b+s)\big]^2+\\
&\qquad\qquad+\big[\psi_1(1/2+b+s)+\psi_1(1/2-b+s)\big]\Biggr\},
\;
s>0,
\end{split}
\end{align}
where here and onward $\psi_{1}(z)\triangleq d\psi_{0}(z)/dz$, i.e., $\psi_{1}(z)$ is the trigamma function, which is a particular case of the more general polygamma function $\psi_{m}(z)\triangleq d^m\psi_{0}(z)/d z^{m}$, $m=0,1,2,\ldots$; for the basic background on the polygamma function, see, e.g.,~\cite[Section~6.4]{Abramowitz+Stegun:Handbook1964}. The substitution $b=1/2$ turns~\eqref{eq:Whit-func-Mellin-2nd-deriv} into
\begin{align}\label{eq:Whit-func-Mellin-2nd-deriv-at-b=one-half}
\int_{0}^{+\infty}
x^{s-1} e^{-\tfrac{x}{2}}\left.\left[\dfrac{\partial^2}{\partial b^2}W_{1,b}(x)\right]\right|_{b=1/2}dx
&=
2\,s\,\Gamma(s)\,\psi_{1}(s),\; s>0,
\end{align}
and the derivation exploits the recurrence $\psi_0(z+1)-\psi_0(z)=1/z$ already used above, and the recurrence $\psi_1(z+1)-\psi_1(z)=-1/z^2$, which is a special case of~\cite[Indetity~6.4.6,~p.~260]{Abramowitz+Stegun:Handbook1964} stating that
\begin{align}\label{eq:polygamma-func-recurrence}
\psi_{m}(z+1)-\psi_{m}(z)
&=
(-1)^m\dfrac{m!}{z^{m+1}},\; m=0,1,2,\ldots.
\end{align}

To ``undo'' the Mellin transform~\eqref{eq:Whit-func-Mellin-2nd-deriv-at-b=one-half}, recall that for any two functions $g_1(x)$ and $g_2(x)$ whose Mellin transforms are $G_1(s)$ and $G_2(s)$, respectively, the Mellin transform of their {\em multiplicative} convolution
\begin{align*}
f(x)
&=
\int_{0}^{+\infty} g_1\left(\dfrac{x}{t}\right)g_2(t)\,\dfrac{dt}{t}
\end{align*}
is the product $G_1(s)\,G_2(s)$. See, e.g.,~\cite[Identity~1.14,~p.~12]{Oberhettinger:Book1974}. Since by~\cite[Identity~4.13,~p.~35]{Oberhettinger:Book1974} the Mellin transform of the function
\begin{align*}
g_2(x)
&=
\dfrac{\log x}{x-1}\,\indicator{0<x<1}
\end{align*}
is precisely $\psi_1(s)$ with $s>0$, i.e., the trigamma function, and because by~\cite[Identity~1.10,~p.~12]{Oberhettinger:Book1974} the Mellin transform of the function $g_1(x)=2\,x\,e^{-x}$, $x\ge0$, is $2\,s\,\Gamma(s)$, also with $s>0$, it follows that
\begin{align*}
\left.\left[\dfrac{\partial^2}{\partial b^2}W_{1,b}(x)\right]\right|_{b=1/2}
&=
2\,e^{\tfrac{x}{2}}\int_{0}^{1}\dfrac{x}{t}\,e^{-\tfrac{x}{t}}\dfrac{1}{t-1}\log t\,\dfrac{dt}{t},
\;
x\ge0.
\end{align*}

The change of integration variables from $t$ to $y\equiv y(t)\triangleq 1/t-1$ allows to rewrite the foregoing as follows:
\begin{align*}
\left.\left[\dfrac{\partial^2}{\partial b^2}W_{1,b}(x)\right]\right|_{b=1/2}
&=
2\,x\,e^{-x}\Biggl\{\int_{0}^{+\infty} e^{-xy}\log(1+y)\,dy+\\
&\qquad\qquad\qquad
+\int_{0}^{+\infty} e^{-xy}\log(1+y)\,\dfrac{dy}{y}\Biggr\},\; x\ge0,
\end{align*}
which, in view of~\eqref{eq:E1-func-def} and~\eqref{eq:MeijerG-func-spcl-case}, can be seen to be precisely part~(\ref{lem:Whit-fcn-deriv:part-B}).

Proving part~(\ref{lem:Whit-fcn-deriv:part-C}) involves exactly the same steps. By differentiating~\eqref{eq:Whit-func-Mellin-2nd-deriv} through with respect to $b$ and then evaluating the result at $b=1/2$ one obtains
\begin{align*}
\int_{0}^{+\infty}x^{s-1}e^{-\tfrac{x}{2}}\left.\left[\dfrac{\partial^3}{\partial b^3}W_{1,b}(x)\right]\right|_{b=1/2}dx
&=
6\,\Gamma(s)\,\psi_1(s),\; s>0,
\end{align*}
and, as before, the key role in the derivation is played by the recurrence~\eqref{eq:polygamma-func-recurrence} with $m$ equal to $0$, $1$, and $2$. This implies that
\begin{align*}
\left.\left[\dfrac{\partial^3}{\partial b^3}W_{1,b}(x)\right]\right|_{b=1/2}
&=
6\,e^{\tfrac{x}{2}}\int_{0}^{1} e^{-\tfrac{x}{t}}\dfrac{1}{t-1}\log t\,\dfrac{dt}{t},\;
x\ge0,
\end{align*}
which, using the substitution $y\equiv y(t)\triangleq 1/t-1$, can be rewritten as
\begin{align*}
\left.\left[\dfrac{\partial^3}{\partial b^3}W_{1,b}(x)\right]\right|_{b=1/2}
&=
6\,e^{-\tfrac{x}{2}}\int_{0}^{+\infty} e^{-xy}\log(1+y)\,\dfrac{dy}{y},\; x\ge0,
\end{align*}
and this, in view of~\eqref{eq:MeijerG-func-spcl-case}, is precisely part~(\ref{lem:Whit-fcn-deriv:part-C}).
\vspace*{1em}
\section*{Acknowledgement}
The author is grateful to the Editor-in-Chief, Nitis Mukhopadhyay (University of Connecticut--Storrs) and to the anonymous referee for the constructive feedback provided on the first draft of the paper that helped improve the quality of the manuscript and shape its current form.

The author's effort was partially supported by the Simons Foundation via a Collaboration Grant in Mathematics under Award \#\,304574.


\end{document}